\documentclass[11pt]{amsart}

\usepackage[a4paper,inner=3.6cm,outer=3.6cm,top=3.7cm,bottom=3.7cm]{geometry}
\usepackage{amsfonts}
\usepackage{amssymb}
\usepackage{amsmath}
\usepackage[all,cmtip,2cell]{xy}
\usepackage{comment}
\usepackage{array}
\usepackage{enumitem}
\usepackage{mathtools}
\usepackage{multirow}
\setlist[itemize]{leftmargin=*}
\setlist[enumerate]{leftmargin=*}

\def\co{\colon\thinspace}
\newcommand{\C}{\mathbb{C}}
\newcommand{\Z}{\mathbb{Z}}
\newcommand{\N}{\mathbb{N}}

\newcommand{\id}{\mathrm{Id}}
\newcommand{\R}{\mathbb{R}}
\newcommand{\chr}{\mathrm{char}\,}

\renewcommand{\P}{\mathbb{P}}
\newcommand{\bd}{\partial}
\newcommand{\ai}{$A_\infty$\ }
\newcommand{\F}{\mathcal{F}}

\newcommand{\M}{\mathcal{M}}

\newcolumntype{C}[1]{>{\centering\arraybackslash$}p{#1}<{$}}

\renewcommand{\S}{S}
\renewcommand{\k}{\mathbb{K}}
\newcommand{\x}{\times}
\newcommand{\CO}{\mathcal{{C}{O}}}
\newcommand{\OC}{\mathcal{{O}{C}}}

\newcommand{\Fuk}{\mathcal{F}uk}

\newcommand{\ev}{\mathrm{ev}}
\renewcommand{\F}{\mathcal{F}}
\newcommand{\D}{\mathcal{D}}
\newcommand{\RP}{\R P}
\newcommand{\CP}{\C P}
\renewcommand{\O}{\mathcal{O}}
\newcommand{\A}{\mathcal{A}}
\newcommand{\B}{\mathcal{B}}
\newcommand{\ot}{\otimes}
\newcommand{\st}{\star}
\DeclareMathOperator{\Spec}{Spec}
\DeclareMathOperator{\Crit}{Crit}
\DeclareMathOperator{\Hess}{Hess}

\newtheorem{theorem}{Theorem}[section]
\newtheorem{proposition}[theorem]{Proposition}
\newtheorem{lemma}[theorem]{Lemma}
\newtheorem{corollary}[theorem]{Corollary}

\newtheorem{hypothesis}[theorem]{Hypothesis}
\newtheorem{conjecture}[theorem]{Conjecture}
\theoremstyle{definition}

\newtheorem{remark}[theorem]{Remark}

\numberwithin{equation}{section}

\begin{document}

\title[The closed-open map for $S^1$-invariant Lagrangians]{The closed-open string map for $S^1$-invariant Lagrangians }
\author{Dmitry Tonkonog}
\email{dt385@cam.ac.uk, dtonkonog@gmail.com}
\address{Department of Pure Mathematics and Mathematical Statistics,
University of Cambridge,
Wilberforce Road,
Cambridge
CB3 0WB, UK}

\begin{abstract}
Given a monotone Lagrangian submanifold invariant under a loop of Hamiltonian diffeomorphisms, we compute a  piece of the closed-open string map into the Hochschild cohomology of the Lagrangian which captures the homology class of the loop's orbit. 

Our applications include split-generation and non-formality results for real Lagrangians in projective spaces and other toric varieties; a particularly basic example is that the equatorial circle on the 2-sphere carries a non-formal Fukaya \ai algebra in characteristic two.
\end{abstract}

\maketitle



\section{Introduction}
\label{sec:intro}
\subsection{Overview of main results}
Let $X$ be a compact monotone symplectic manifold, $L\subset X$  a monotone Lagrangian submanifold, and $\k$ be a field. 
We assume that $L$ satisfies the usual conditions making its Floer theory well-defined over $\k$, namely, $L$ has Maslov index at least 2, and is oriented and spin if $\chr\k\neq 2$. In this case, one can define a unital algebra over $\k$, the Floer cohomology $HF^*(L,L)$, which is invariant under Hamiltonian isotopies of $L$. A larger amount of information about $L$ is captured by the Fukaya \ai algebra of $L$, and given this \ai algebra, one can build another associative unital algebra called the Hochschild cohomology $HH^*(L,L)$. There is the so-called (full) closed-open string map
$$\CO^*\co QH^*(X)\to HH^*(L,L),$$
which is a map of unital algebras, where $QH^*(X)$ is the (small) quantum cohomology of $X$. 
This map is of major importance in symplectic topology,
particularly in light of Abouzaid's split-generation criterion \cite{Ab10,She13,RiSmi17,AFO3}, one of whose versions in the case $\chr\k=2$ says the following: if the closed-open map is {\it injective}, then $L$ split-generates the $w$-summand $\Fuk(X)_w$ of the Fukaya category, where $w=w(L)\in \k$ is the so-called obstruction number of $L$.  (When $\chr\k\neq 2$, the hypothesis can be weakened to say that $\CO^*$ is injective on a relevant eigensummand of $QH^*(X)$; we will recall this  later.)

Split-generation of  the Fukaya category $\Fuk(X)_w$ by a Lagrangian submanifold~$L$ is an algebraic phenomenon which has important geometric implications. For example, in this case $L$ must have non-empty intersection with any other monotone Lagrangian submanifold $L'$ which is a non-trivial object in $\Fuk(X)_w$, namely such that $HF^*(L',L')\neq 0$ and $w(L')=w$. Another application, though not discussed here, is that  split-generation results are used in proofs of homological mirror symmetry.

The present paper contributes with new calculations of the closed-open map, motivated by the split-generation criterion and the general lack of explicit calculations known so far. (The closed-open map is defined by counting certain pseudo-holomorphic disks with boundary on $L$, which makes it extremely hard to compute in general.)

There is a simplification of the full closed-open map, called the ``zeroth-order'' closed-open map, which is a unital algebra map
$$
\CO^0\co QH^*(X)\to HF^*(L,L).
$$
It is the composition of $\CO^*$ with the canonical projection $HH^*(L,L)\to HF^*(L,L)$, and if $\CO^0$ is injective, so is $\CO^*$ (but not vice versa). Although $\CO^0$ generally carries less information than $\CO^*$, it is  sometimes easier to compute. For example, we compute $\CO^0$ when $L$ is the real locus of a complex toric Fano variety $X$, see Theorem~\ref{th:toric_CO_0}. This map turns out to be non-injective in many cases, e.g.~for $\RP^{2n+1}\subset \CP^{2n+1}$ over a characteristic 2 field. The aim of the present paper is to study the higher order terms of the full closed-open map $\CO^*$, and to find examples when $\CO^*$ is injective but $\CO^0$ is not.

Specifically, let us consider the following setting: a loop  $\gamma$ of Hamiltonian symplectomorphisms preserves a Lagrangian $L$ setwise. Let $\S(\gamma)\in QH^*(X)$ be the Seidel element of $\gamma$, then from Charette and Cornea~\cite{CC16} one can see that 
$$\CO^0(\S(\gamma))=1_L,$$ the unit in $HF^*(L,L)$.
Our main result, Theorem~\ref{th:CO_invt_lag}, is a tool for distinguishing $\CO^*(\S(\gamma))$ from the Hochschild cohomology unit in $HH^*(L,L)$; this way it captures a non-trivial piece of the full closed-open map $\CO^*$ not seen by $\CO^0$. 
We apply Theorem~\ref{th:CO_invt_lag} to show that $\CO^*$ is injective for some real Lagrangians in toric manifolds, and also for monotone toric fibres which correspond to (non-Morse) $A_2$-type critical points of the Landau-Ginzburg superpotential. 

After this paper had appeared, Evans and Lekili \cite{EL15} proved split-generation for all orientable real toric Lagrangians, and all monotone toric fibres in zero characteristic, by completely different methods. They make use of the fact that these are homogeneous Lagrangians (i.e.~ they are orbits of Hamiltonian group actions), while we only use the fact these Lagrangians are invariant under certain Hamiltonian loops.

We will now mention our examples regarding real Lagrangians, and postpone all discussion of monotone toric fibres, along with an introductory part, to Section~\ref{sec:toric_fibre}.
 
\begin{proposition}
\label{prop:CO_RPn_Inject}
Let $\k$ be a field of characteristic $2$ and $\RP^n$ be the standard real Lagrangian in $\CP^n$.
Then $\CO^*\co QH^*(\CP^n)\to HH^*(\RP^n,\RP^n)$ is injective for all $n$.  
In contast, 
$\CO^0\co QH^*(\C P^n)\to HF^*(\RP^n,\RP^n)$ is injective if and only if $n$ is even.
\end{proposition}

\begin{corollary}
\label{cor:RPn_generate}
Over a field of characteristic $2$, $\RP^n$ split-generates $\Fuk(\CP^n)_0$.
\end{corollary}

As hinted above, this corollary leads to a result on non-displaceability of $\RP^n$ from other monotone  Lagrangians which are Floer-theoretically non-trivial. This has been known due to Biran and~Cornea~\cite[Corollary~8.1.2]{BC09B}, and Entov and Polterovich~\cite{EnPo09}. 
Very recently Konstantinov~\cite{Kon17} showed that the Chiang Lagrangian in $\CP^3$ admits a higher rank local system making it Floer-theoretically non-trivial over a characteristic~two field; then he concludes via Corollary~\ref{cor:RPn_generate} that the Chiang Lagrangian is non-displaceable from $\RP^3$. It is possible that for Lagrangians with higher rank local systems, a generalisation of \cite{BC09B, EnPo09} can be invoked instead of Corollary~\ref{cor:RPn_generate}, but we have not checked this. 

We can extract another interesting consequence about projective spaces from our main computation of the closed-open map.

\begin{proposition}
\label{prop:RPn_not_formal}
The Fukaya \ai algebra of the Lagrangian $\RP^{4n+1}\subset \CP^{4n+1}$ is not formal over a characteristic 2 field, for any $n\ge 0$.
\end{proposition}

Here, formality means an existence of a quasi-isomorphism with the associative algebra $HF^*(\RP^{4n+1},\RP^{4n+1})\cong \k[u]/(u^{4n+2}-1)$, considered as an \ai algebra with trivial higher-order structure maps.
In particular, the Fukaya \ai algebra of the equator  $S^1\subset S^2$ is not formal in characteristic 2; we devote a separate discussion to this fact in Section~\ref{sec:real_toric} where explicitly exhibit a non-trivial Massey product which provides an alternative proof of the non-formality.
Below is another example of split-generation which we can prove using the same methods.

\begin{proposition}
\label{prop:CO_BlP9_Inject}
Let $\k$ be a field of characteristic $2$, $X=Bl_{\CP^1}\CP^9$ the blow-up of $\CP^9$ along a complex line which intersects $\RP^9$ in a circle, and let $L\subset X$ be the blow-up of $\RP^9$ along that circle. 
Then $\CO^*\co QH^*(X)\to HH^*(L,L)$ is injective although $\CO^0\co QH^*(X)\to HF^*(L,L)$ is not. 
Consequently, $L$ split-generates $\Fuk(X)_0$.
\end{proposition}

(The manifold $Bl_{\CP^1}\CP^9$ is the first instance among $Bl_{\CP^k}\CP^n$ for which $L$ is monotone of Maslov index at least $2$, and such that $\CO^0$ is not injective --- the last requirement makes the use of our general results  essential in this example.)
In general, it is known that the real Lagrangian in a toric Fano variety is not displaceable from the monotone toric fibre: this was proved by Alston and Amorim \cite{AA11}. Proposition~\ref{prop:CO_BlP9_Inject} implies a much stronger non-displaceability result, like the one which has been known for $\RP^n\subset \CP^n$.

\begin{corollary}
\label{cor:non_disp_Bl_P9}
 Let $\k$ and $L\subset X$ be as in Proposition~\ref{prop:CO_BlP9_Inject}, and $L'\subset X$ any other monotone Lagrangian, perhaps equipped with a local system $\pi_1(L)\to \k^\x$, with minimal Maslov number at least 2 and such that $HF^*(L',L')\neq 0$. If $w(L')\neq 0$, we also assume the technical Hypothesis~\ref{hyp:ganatra}, which is expected to hold following \cite{Ga13}. Then $L\cap L'\neq\emptyset$. 
\end{corollary}

Here $HF^*(L',L')$ denotes the Floer cohomology of $L'$ with respect to the local system $\rho$, so a better notation would be $HF^*((L',\rho),(L',\rho))$. For brevity, we decided to omit $\rho$ from our notation of Floer and Hochschild cohomologies throughout the article, when it is clear that a Lagrangian is equipped with such a local system. The point of allowing local systems in Corollary~\ref{cor:non_disp_Bl_P9} is to introduce more freedom in achieving the non-vanishing of $HF^*(L',L')$.

Note that Corollary~\ref{cor:non_disp_Bl_P9} does not require that the obstruction number of $L'$ matches the one of $L$, namely zero. If $w(L')\neq 0$, we can pass to $X\times X$  noticing that
$w(L'\times L')=2w(L')=0$ and similarly $w(L\times L)=0$, so we have well-defined Floer theory between the two product Lagrangians. 
This trick was observed by Abreu and Macarini~\cite{AM13}
and has also been used in \cite{AA11}. 
So it suffices to show that $L\times L$ split-generates $\Fuk(X\times X)_0$; this follows from Proposition~\ref{prop:CO_BlP9_Inject} by the general expectation that the condition of the Abouzaid's split-generation criterion is ``preserved'' under K\"unneth isomorphisms. As we explain later, this general expectation is contingent upon a certain commutative diagram which we formulate as Hypothesis~\ref{hyp:ganatra}, and which is largely substantiated by Ganatra \cite{Ga13}; see also \cite{AS10,Am14}.

As in the case with $\RP^n$, we  also prove a non-formality statement.

\begin{proposition}
\label{prop:BlP9_not_formal}
The Fukaya \ai algebra of the Lagrangian $Bl_{\RP^1}\RP^9\subset BL_{\CP^1}\CP^9$ from Proposition~\ref{prop:CO_BlP9_Inject} is not formal over a characteristic 2 field.
\end{proposition}

Although we cannot prove that $\CO^*$ is injective for the real locus of an arbitrary toric Fano variety, we are able to do this in a slightly wider range of examples  which we postpone to Section~\ref{sec:real_toric}.
We will prove Proposition~\ref{prop:CO_RPn_Inject} and Corollary~\ref{cor:RPn_generate}  at the end of the introduction, and the remaining statements from above will be proved in Section~\ref{sec:real_toric}.
Now we state the main theorem; the new pieces of notation are explained straight after the statement.

\begin{theorem}
\label{th:CO_invt_lag}
Let $X$ be a compact monotone symplectic manifold, $L\subset X$ a monotone Lagrangian submanifold of Maslov index at least $2$, possibly equipped with a local system $\rho\co H_1(L)\to \k^\x$. If $\chr \k\neq 2$, assume $L$ is oriented and spin. 

Let
$\gamma=\{\gamma_t\}_{t\in S^1}$ be a loop of Hamiltonian symplectomorphisms of $X$, and denote by $\S(\gamma)\in QH^*(X)$ the corresponding Seidel element. 
Suppose the loop $\gamma$ preserves $L$ setwise, that is, $\gamma_t(L)=L$. Denote by $l\in H_1(L)$ the homology class
of an orbit $\{\gamma_t(q)\}_{t\in S^1}$, $q\in L$. Finally, assume $HF^*(L,L)\neq 0$.
\begin{enumerate}
 \item[(a)] Then $\CO^0(\S(\gamma))=(-1)^{\epsilon(l)}\cdot \rho(l)\cdot1_L$ where $1_L\in HF^*(L,L)$ is the unit.
\item[(b)] 
Suppose there exists no $a\in HF^*(L,L)$ such that
\begin{multline*}
\tag*{$(*)$}
\mu^2(a,\Phi(y))+\mu^2(\Phi(y),a)=(-1)^{\epsilon(l)}\rho(l)\cdot\langle y,l\rangle\cdot 1_L
\quad\text{for\ each\ } y\in H^1(L).
\end{multline*}
Then $\CO^*(\S(\gamma))\in HH^*(L,L)$ is linearly independent from the Hochschild cohomology unit.
\item[(c)]
More generally, suppose $Q\in QH^*(X)$ and there exists no $a\in HF^*(L,L)$ such that
\begin{equation*}
\tag{$**$}
\mu^2(a,\Phi(y))+\mu^2(\Phi(y),a)=(-1)^{\epsilon(l)}\rho(l)\cdot\langle y,l\rangle\cdot \CO^0(Q) \quad \text{for\ each\ } y\in H^1(L).
\end{equation*}
Then $\CO^*(\S(\gamma)*Q)$ and $\CO^*(Q)$ are linearly independent in the Hochschild cohomology $HH^*(L,L)$.
\end{enumerate}
\end{theorem}

Here $\mu^2$ is the product on $HF^*(L,L)$, $\langle-,-\rangle$ is the pairing $H^1(L)\otimes H_1(L)\to\k$, and $\S(\gamma)*Q$ is the quantum product of the two elements. Next, 
$$\Phi\co H^1(L)\to HF^*(L,L)$$ 
is the PSS map of Albers \cite{Alb08}, which is canonical and well-defined if $HF^*(L,L)\neq 0$. Its well-definedness in a setting closer to ours was studied by e.g.~Biran and Cornea \cite{BC09A}, and later we discuss it in more detail. Note that $\Phi$ is not necessarily injective, although in our applications, when $HF^*(L,L)\cong H^*(L)$, it will be. Finally, in the theorem we have allowed $L$ to carry an arbitrary local system, which modifies the Fukaya \ai structure of $L$ by counting the same
punctured holomorphic disks as in the case without a local system with coefficients which are the values of $\rho$ on the boundary loops of such disks. The algebras $HF^*(L,L)$, $HH^*(L,L)$ get modified accordingly, although their dependence on $\rho$ is not reflected by our notation, as mentioned earlier. We allow non-trivial local systems in view of our application to toric fibres, and will only need the trivial local system $\rho\equiv 1$ for applications to real Lagrangians.

To complete the statement of Theorem~\ref{th:CO_invt_lag}, we need to explain the sign $(-1)^{\epsilon(l)}=\pm 1$ appearing in it. By a spin Lagrangian, we always mean a Lagrangian with a fixed spin structure (rather than admitting one). We have two natural trivialisations of $TL$ over the loop ${\gamma_t(q)}\subset L$: the one induced from a fixed basis of $T_qL$ by the Hamiltonian loop $\gamma$, and the one determined by the spin structure on $L$. We put $\epsilon(l)$ to be~0 if the two trivialisations agree, and~1 otherwise.

\smallskip
{\it Outline of proof.} It has been mentioned  earlier that part (a) of Theorem~\ref{th:CO_invt_lag} is an easy consequence of the paper by Charette and Cornea~\cite{CC16}.
The proof of parts~(b) and~(c) also starts by using a result from that paper, and then the main step is an explicit computation of $$\CO^1(\S(\gamma))|_{CF^1(L,L)}\co CF^1(L,L)\to CF^0(L,L)$$ on cochain level, which turns out to be dual to taking the $\gamma$-orbit of a point up to the factor $(-1)^{\epsilon(l)}\rho(l)$: this is Proposition~\ref{prop:compute_CO_1}. The final step is to check whether the computed nontrivial piece of the Hochschild cocycle $\CO^*(\S(\gamma))$ survives to cohomology; this is controlled by  equations~$(*)$,~$(**)$.

\begin{remark}
\label{rem:toric_positive}
In our examples
we will never encounter a non-trivial sign $(-1)^{\epsilon(l)}$: for real Lagrangians we shall be working over characteristic~two, and for toric fibres with the standard spin structure, this sign is easily seen to be $+1$. The examples when the sign $(-1)^{\epsilon(l)}$ is negative have been found by J.~Smith; they occur for $PSU(N-1)$-homogeneous Lagrangians \cite[Remark~5.3.2]{Smi17}. We are grateful to him for pointing out the presence of this sign in general, which was missed in the previous versions of the paper.
\end{remark}

\subsection{The split-generation criterion}
We will now briefly discuss the split-gener\-ation criterion in more detail, particularly because we wish to pay attention to both $\chr \k=2$ and $\chr\k\neq 2$ cases. We continue to denote by $L\subset X$ a monotone Lagrangian submanifold with minimal Maslov number at least 2, which is oriented and spin if $\chr \k\neq 2$. If $\chr\k=2$, we allow $L$ to be non-orientable.
Consider the quantum multiplication by the first Chern class as an endomorphism of quantum cohomology, $-*c_1(X)\co QH^*(X)\to QH^*(X)$. If $\k$ is algebraically closed, we have an algebra decomposition $QH^*(X)=\oplus_w QH^*(X)_w$ where $QH^*(X)_w$ is the generalised $w$-eigenspace of $-*c_1(X)$, $w\in \k$.

Recall that $w(L)\in \k$ denotes the obstruction number of $L$, i.e.~the count of Maslov index~2 disks with boundary on $L$. By an observation of Auroux, Kontsevich and Seidel, $\CO^0(2c_1)=2w(L)\cdot 1_L$, which in $\chr \k\neq 2$ implies that $\CO^0(c_1)=w(L)\cdot 1_L$, see e.g.~\cite{She13}. Now suppose that $\chr\k=2$ and $c_1(X)$ lies in the image of $H^2(X,L;\k)\to H^2(X;\k)$, which is true if $L$ is orientable (because the Maslov class goes to twice the Chern class under $H^2(X,L;\Z)\to H^2(X;\Z)$, and the Maslov class of an orientable manifold is integrally divisible by two). In this case, the same argument shows again that $\CO^0(c_1)=w(L)\cdot 1_L$.
 This way one deduces the following lemma, which is well-known but usually stated only for $\chr \k\neq 2$.

\begin{lemma}
\label{lem:eigenvalue_split}
 For $\k$ of any characteristic, if $L$ is orientable, then $\CO^0\co QH^*(X)\to HF^*(L,L)$ vanishes on all summands except maybe $QH^*(X)_{w(L)}$.\qed
\end{lemma}
(If $w(L)$ is not an eigenvalue of $-*c_1(X)$, then $\CO^0$ vanishes altogether, and it follows that $HF^*(L,L)=0$. Recall that $L$ is required to be monotone.) The same vanishing statement is expected to hold for the full map $\CO^*$.
Keeping this  vanishing in mind, we see that the ``naive'' version of the split-generation criterion stated in the introduction, that $\CO^*\co QH^*(X)\to HH^*(L,L)$ is injective, can only be useful when $\chr\k=2$ and $L$ is non-orientable. In other cases it must be replaced by a more practical criterion which does not ignore the eigenvalue decomposition; we will now state both versions of the criterion.
Let $\F(X)_{w}$ denote the Fukaya category whose objects are monotone Lagrangians in $X$ with minimal Maslov number at least 2, oriented and spin if $\chr \k \neq 2$, and  whose obstruction number equals $w\in \k$. 

\begin{theorem}
\label{th:generation_crit}
Let $L_1,\ldots,L_n\subset X$ be Lagrangians which are objects of $\Fuk(X)_w$, and $\mathcal{G}\subset \Fuk(X)_w$ be the full subcategory generated by $L_1,\ldots, L_n$. Then $\mathcal{G}$ split-generates $\Fuk(X)_w$ if either of the two following statements hold.

\begin{enumerate}
 \item[(a)] $\chr\k\neq 2$,
and $\CO^*|_{QH^*(X)_w}\co QH^*(X)_w\to HH^*(\mathcal{G})$ is injective.

\item[(b)] $\k$ is arbitrary, and $\CO^*\co QH^*(X)\to HH^*(\mathcal{G})$ is injective.\qed
\end{enumerate}
\end{theorem}
In the monotone case, this theorem is due to Ritter and Smith \cite{RiSmi17} and Sheridan \cite{She13}. It is more common to only state part (a), but it is easy to check the same proof works for part (b) as well. (In part (a), we could also allow $\chr\k=2$, if $L$ is orientable.)
Theorem~\ref{th:generation_crit} is most easily applied when  $QH^*(X)_{w}$ is 1-dimensional: because $\CO^*$ is unital, it  automatically becomes injective. We  are going to apply this theorem in more complicated cases. Before we proceed, let us mention one easy corollary of split-generation. We say that $L_1,\ldots,L_n$ split-generate the Fukaya category when $\mathcal{G}$ does.

\begin{lemma}
\label{lem:split_gen_implies_non_disp}
 If Lagrangians $L_1,\ldots,L_n\subset X$ split-generate $\Fuk(X)_w$, and $L\subset X$ is another Lagrangian which is an object of $\Fuk(X)_w$ with $HF^*(L,L)\neq 0$, then $L$ has non-empty intersection, and non-zero Floer cohomology, with one of the Lagrangians $L_i$.\qed
\end{lemma}

\subsection{$\CO^0$ for real toric Lagrangians}
In this subsection we state a theorem that computes $\CO^0$ for real Lagrangians in toric manifolds. Using it, it is easy to identify the cases when $\CO^0$ is injective (and the split-generation follows immediately), and the cases when $\CO^0$ is not injective and therefore a further study of $\CO^*$ is required to establish the split-generation. Our subsequent goal is to apply the main result, Theorem~\ref{th:CO_invt_lag}, to some examples of the later type.

Let $X$ be a (smooth, compact) toric Fano variety with minimal Chern number at least 2, i.e.~$\langle c_1(X),H_2(X;\Z)\rangle=N\Z$, $N\ge 2$. As a toric manifold $X$ has a canonical anti-holomorphic involution $\tau\co X\to X$. Its fixed locus is the so-called real Lagrangian $L\subset X$ which is smooth \cite[p.~419]{Du83}, monotone and whose minimal Maslov number equals the minimal Chern number of $X$ \cite{Ha13}. When speaking of such real Lagrangians, we will always be working over a field $\k$ of characteristic $2$. In particular, there is the Frobenius map:
$$\F\co QH^*(X)\to QH^{2*}(X), \quad \F(x)=x^2.$$ 
Because $\chr \k=2$, $\F$ is a map of unital algebras. We have reflected in our notation that $\F$ multiplies the $\Z/2N$-grading by two. A classical theorem of Duistermaat \cite{Du83} constructs, again in $\chr \k=2$, the isomorphisms $H^i(L)\cong H^{2i}(X)$. We can package these isomorphisms into a single isomorphism of unital algebras,
$$
\D\co H^{2*}(X)\xrightarrow{\cong} H^*(L).
$$
Let us now recall a recent theorem of Haug \cite{Ha13}.
\begin{theorem}
\label{th:haug}
 If $\chr \k=2$, then  $HF^*(L,L)\cong H^*(L)$ as vector spaces. Using the identification coming from a specific perfect Morse function from~\cite{Ha13}, and also indentifying $QH^*(X)\cong H^*(X)$, the same map
$$
\D\co QH^{2*}(X)\xrightarrow{\cong}HF^*(L,L)
$$
is again an isomorphism of unital algebras.\qed
\end{theorem}

It turns out that it is possible to completely compute $\CO^0$ for real toric Lagrangians.
This rather quickly follows by combining the works of  Charette and Cornea~\cite{CC16}, Hyvrier~\cite{Hy16}, and McDuff and Tolman \cite{MDT06}; we explain this theorem in Section~\ref{sec:real_toric}.

\begin{theorem}
\label{th:toric_CO_0}
The diagram below commutes.
$$
\xymatrix{
 QH^*(X)\ar[r]^{\F} \ar[dr]_{\CO^0} &QH^{2*}(X)\ar[d]_-\D^-\cong\\
& HF^*(L,L)
}
$$
\end{theorem}
In particular, $\CO^0$ is injective if and only if $\F$ is injective.

\subsection{Split-generation for the real projective space}
We conclude the introduction by proving Proposition~\ref{prop:CO_RPn_Inject} and Corollary~\ref{cor:RPn_generate}.
The crucial idea is that when $n$ is odd, the kernel of $\CO^0\co QH^*(\CP^n)\to HF^*(\RP^n,\RP^n)$  is the ideal generated by the Seidel element of a non-trivial  Hamiltonian loop preserving $\RP^n$; this allows to apply Theorem~\ref{th:CO_invt_lag} and get new information about $\CO^*$.
 Recall that $QH^*(X)\cong\k[x]/(x^{n+1}-1)$ and $w(\RP^n)=0$, because the minimal Maslov number of $\RP^n$ equals $n+1$ (when $n=1$, we still have $w(S^1)=0$ for $S^1\subset S^2$).

\begin{proof}[Proof of Proposition~\ref{prop:CO_RPn_Inject}.]
If $n$ is even, the Frobenius map on $QH^*(\CP^n)$ is injective, so by Theorem~\ref{th:toric_CO_0}, $\CO^0\co QH^*(\CP^n)\to HF^*(\RP^n,\RP^n)$ is injective, and hence $\CO^*$ too.  

Now suppose $n$ is odd and denote $n=2p-1$. Given $\chr\k=2$, we have $QH^*(\CP^n)\cong\k[x]/(x^p+1)^2$, so $\ker \F=\ker \CO^0$ is the ideal generated by $x^p+1$. Consider the Hamiltonian loop $\gamma$ on $\CP^n$ which in homogeneous co-ordinates $(z_1:\ldots: z_{2p})$ is the rotation
$\left( \begin{smallmatrix}
  \cos t&\sin t\\
-\sin t&\cos t
 \end{smallmatrix}
\right)$, $t\in [0,\pi]$,
applied simultaneously to the pairs $(z_1,z_2),\ldots,(z_{2p-1},z_{2p})$. Note that $t$ runs to $\pi$, not $2\pi$.
This loop is Hamiltonian isotopic to the loop
$$
(z_0:\ldots:z_{2p-1})\mapsto (e^{2it}z_0:z_1:\ldots:e^{2it}z_{2p-1}:z_{2p}), \quad t\in[0,\pi],
$$
so $\S(\gamma)=x^p$, see \cite{MDT06}. The loop $\gamma$ obviously preserves the real Lagrangian $\RP^n\subset \CP^n$, and its orbit $l$ is a generator of $H_1(\RP^n)\cong \k$. Taking $y\in H^1(\RP^n)$ to be the generator, we get $\langle y,l\rangle =1$, and the right hand side of equation $(**)$ from Theorem~\ref{th:CO_invt_lag}  equals $\CO^0(Q)$. On the other hand, the product on $HF^*(\RP^n,\RP^n)$ is commutative by Theorem~\ref{th:haug}, so the left hand side of $(**)$ necessarily vanishes.  We conclude that the hypothesis of Theorem~\ref{th:CO_invt_lag}(c) is satisfied for any $Q\notin\ker\CO^0$. 

Let us prove that $\CO^*(P)\neq 0$ for each nonzero $P\in QH^*(\CP^n)$. If $\CO^0(P)\neq 0$, we are done, so it suffices to
suppose that $\CO^0(P)=0$. It means that 
$$P=(x^p+1)*Q=(\S(\gamma)+1)*Q$$ 
for some $Q\in QH^*(\CP^n)$. Note that if $Q\in \ker \CO^0=\ker \F$ then $P\in (\ker \F)^2=\{0\}$. So if $P\neq 0$, then $\CO^0(Q)\neq 0$, and thus $\CO^*(P)\neq 0$  by Theorem~\ref{th:CO_invt_lag}(c) and the observation earlier in this proof.
\end{proof}

\begin{remark}
When $n$ is even, $c_1(\CP^n)$ is invertible in $QH^*(\CP^n)$, so the 0-eigenspace $QH^*(\CP^n)_0$ is trivial; but $L$ is non-orientable, so this does not contradict Lemma~\ref{lem:eigenvalue_split}.
On the other hand, when $n$ is odd, $L$ is orientable but $c_1(\CP^n)$ vanishes in $\chr \k=2$, so the whole $QH^*(\CP^n)$ is its 0-eigenspace; this is also consistent with Lemma~\ref{lem:eigenvalue_split}.
\end{remark}

\begin{proof}[Proof of Corollary~\ref{cor:RPn_generate}]
 This follows from Proposition~\ref{prop:CO_BlP9_Inject} and Theorem~\ref{th:generation_crit}(b).
\end{proof}

The same trick of finding a real Hamiltonian loop whose Seidel element generates $\ker \CO^0$
works for some other toric manifolds which have ``extra symmetry'' in addition to the toric action, like a Hamiltonian action of $SU(2)^{\dim_\C X/2}$ which was essentially used above. As already mentioned, we will provide more explicit examples in Section~\ref{sec:real_toric}.

\subsection*{Acknowledgements}
The author is most grateful to his supervisor Ivan Smith for many useful comments as well as constant care and enthusiasm. Yank\i~Lekili and Jack~Smith have provided valuable feedback and pointed out two inaccuracies in the previous versions of this paper.
The paper has also benefitted from discussions with Mohammed Abouzaid, Lino Amorim, Fran\c cois Charette, Georgios Dimitroglou Rizell, Jonny Evans, Alexander Ritter, Paul Seidel, Nick Sheridan and Renato Vianna. The referee's suggestions on improving the exposition have been very useful.

The author was funded by the Cambridge Commonwealth, European and International Trust, and acknowledges travel funds from King's College, Cambridge.

\section{Proof of Theorem~\ref{th:CO_invt_lag}}
\label{sec:main_proof}
Let $X$ be a monotone symplectic manifold and $w\in \k$. We recall that the objects in the monotone Fukaya category $\Fuk(X)_w$  are monotone Lagrangian submanifolds $L\subset X$ with minimal Maslov number at least 2, oriented and spin if $\chr\k\neq 2$,
equipped with local systems $\rho\co \pi_1(L)\to \k^\x$, whose count of Maslov 2 disks (weighted using $\rho$) equals $w$. We will use the definition of the Fukaya category based on achieving transversality by explicit Hamiltonian perturbations of the pseudo-holomorphic equation. This setup was developed by Seidel \cite{SeiBook08} for exact manifolds and carries over to monotone ones, see \cite{RiSmi17, She13, AFO3}.
There is a notion of bounding cochains from \cite{FO3Book}, generalising the notion of a local system, and all results are expected carry over to them as well.

\subsection{A theorem of Charette and Cornea}
Suppose $\gamma=\{\gamma_t\}_{t\in S^1}$ is a loop of Hamiltonian symplectomorphisms on $X$.
As explained by Seidel in \cite[Section (10c)]{SeiBook08}, the loop $\gamma$ gives rise to a natural transformation 
$\gamma^\sharp$ from the identity functor on $\Fuk(X)_w$ to itself. Any such natural transformation  is a cocycle of the Hochschild cochain complex $CC^*(\Fuk(X)_w)$ \cite[Section (1d)]{SeiBook08}. Denote the corresponding Hochschild cohomology class by 
$$[\gamma^\sharp]\in HH^*(\Fuk(X)_w).$$
We denote, as earlier, the closed-open map by $\CO^*\co QH^*(X)\to HH^*(\Fuk(X)_w)$ and the Seidel element by $\S(\gamma)\in QH^*(X)$.
The following theorem was proved by Charette and Cornea \cite{CC16}.

\begin{theorem}
\label{th:Charette_Cornea}
If we take for $\Fuk(X)_w$ the Fukaya category of Lagrangians with trivial local systems only, then 
$\CO^*(\S(\gamma))=[\gamma^\sharp]$. \qed
\end{theorem}

Let us now restrict to a single Lagrangian $L$ which is preserved by the Hamiltonian loop $\gamma$, and denote by $l\in H_1(L)$ the homology class of an orbit of $\gamma$ on $L$. 
Let $CC^*(L,L)$ denote the Hochschild cochain complex of the \ai algebra $CF^*(L,L)$, and let $HH^*(L,L)$ be its Hochschild cohomology. (The definition of Hochschild cohomology will be reminded later in this section.)
We will now need to recall the proof of Theorem~\ref{th:Charette_Cornea} for
several reasons: first, we wish to see how Theorem~\ref{th:Charette_Cornea} gets modified in the presence of a local system on $L$; second, we shall see the appearance of a sign ``hidden'' in $\gamma^\sharp$; third and most importantly,
we will recall the definition of the moDuli spaces computing $\gamma^\sharp$ in the process. Eventually, for later use we  need a form of Theorem~\ref{th:Charette_Cornea} expressed by formula (\ref{eq:CO}) below, which takes the local system and the sign into account.

Pick some Floer datum $\{H_s,J_s\}_{s\in [0,1]}$ and perturbation data defining an \ai structure on Floer's complex $CF^*(L,L)$ \cite{SeiBook08}.
Recall that the maps $$\CO^k(\S(\gamma))\co CF^*(L,L)^{\otimes k}\to CF^*(L,L)$$ 
count 0-dimensional moduli space of  disks satisfying a perturbed pseudo-holomorphic equation (with appropriately chosen perturbation data) with $k+1$ boundary punctures ($k$ inputs and one output) and one interior marked point. These disks satisfy the Lagrangian boundary condition $L$, and their interior marked point is constrained to  a  cycle  dual to $\S(\gamma)$, see Figure~\ref{fig:loop_spread}(a) (in this figure, we abbreviate the datum $\{H_s,J_s\}$ simply to $H$). A disk $u$ is counted with coefficient $\pm\rho(\bd u)$ where the sign $\pm$ comes from the orientation on the moduli space and $\rho(\bd u)\in\k^\x$ is the monodromy of the local system.
The collection of maps $\CO^*(\S(\gamma))\coloneqq\{\CO^k\}_{k\ge 0}$ is a Hochschild cochain in $CC^*(L,L)$, if all perturbation data are chosen consistently with gluing. 

\begin{figure}[h]
\centerline{ \includegraphics{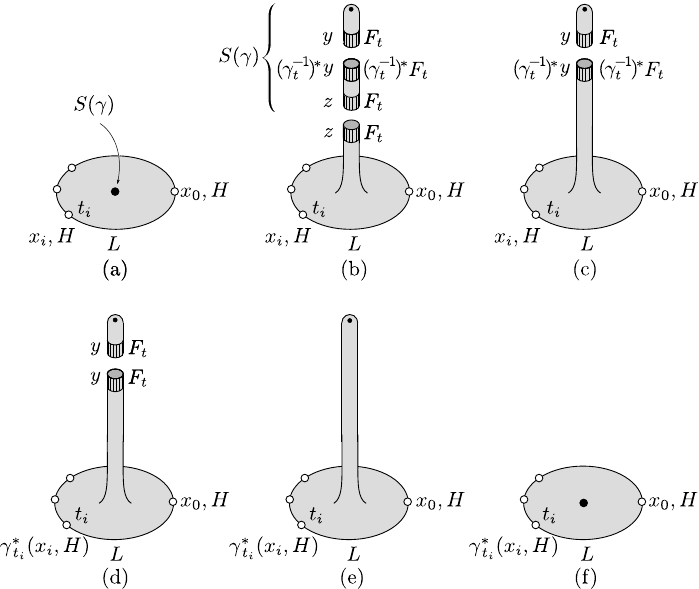} }
\caption{A computation of $\CO^*(\S(\gamma))$ by Charette and Cornea.}
\label{fig:loop_spread}
\end{figure}

The argument of Charette and Cornea
starts by passing to a more convenient definition of the closed-open map in which $\CO^k$ count holomorphic disks with $k+1$ boundary punctures and one interior puncture (instead of a marked point). We can view the neighbourhood of the interior puncture as a semi-infinite cylinder, then the pseudo-holomorphic equation  restricts on this semi-infinite cylinder to a Hamiltonian Floer equation with some Floer datum $\{F_t,J_t\}_{t\in S^1}$. We input the PSS image of $\S(\gamma)$ to the interior puncture, see Figure~\ref{fig:loop_spread}(b), given as a linear combination of some Hamiltonian orbits $z$ (in the figure, we abbreviate the datum $\{F_t,J_t\}$ simply to $F_t$). 

The PSS image of $\S(\gamma)$ counts configurations shown in the upper part of Figure~\ref{fig:loop_spread}(b), consisting of disks with one output puncture (say, asymptotic to an orbit $y$), and a cylinder counting continuation maps from $(\gamma_t^{-1})^*y$ seen as an orbit of Floer's complex with datum pulled back by the loop $\gamma_{t}^{-1}$ \cite[Lemmas 2.3 and~4.1]{Sei97}, to another orbit $z$ of the original Floer's complex with datum $\{F_t,J_t\}$.
Let us glue the $z$-orbits together, passing to  Figure~\ref{fig:loop_spread}(c), and then substitute each lower punctured pseudo-holomorphic disk $u$ in Figure~\ref{fig:loop_spread}(c) by $\tilde u$ defined as follows: 
\begin{equation}
\label{eq:u_substitution}
 \tilde u(re^{2\pi it})=\gamma_{t}^{-1}\circ u(re^{2\pi it}),
 \end{equation}
assuming that the interior puncture is located at $0\in\C$ and the output puncture at $1\in\C$. Let us look at the effect of this substitution. 

First, $[\bd u]=[\bd \tilde u]+l\in H_1(L)$ so the count of configurations in Figure~\ref{fig:loop_spread}(c) (before substitution) is equal to the count of configurations in Figure~\ref{fig:loop_spread}(d) (after substitution) multiplied by $\rho(l)$. 

Second, $\tilde u$ satisfies the same boundary condition $L$ because $\gamma_tL=L$, but the perturbation data defining the pseudo-holomorphic equation get pulled back accordingly. In particular, the Lagrangian Floer datum $\{H_s,J_s\}_{s\in [0,1]}$ and the asymptotic chord at a strip-like end corresponding to the boundary puncture at $t_i\in S^1$ get pulled back by $\gamma_{t_i}$.

Third, (\refeq{eq:u_substitution}) gives an abstract bijection $u\mapsto \tilde u$ between the respective zero-dimensional moduli spaces,
but we should discuss how this bijection behaves with respect to the signs attached to $u,\tilde u$ by the orientations on the moduli spaces. 
Assume for simplicity that the Hamiltonian perturbation is small enough, so that we can canonically deform $\bd u$ and $\bd \tilde u$ to loops inside $L$. Take the trivialisation of $TL|_{\bd u}$ defined by the spin structure on $L$, and push it forward by $\gamma$ to a trivialisation of $TL|_{\bd \tilde u}$. Obviously, \eqref{eq:u_substitution} preserves the signs computed using these trivialisations; we remark that such trivialisation of $TL|_{\bd \tilde u}$ is used in the general definition of $\gamma^\sharp$.
In our specific case,  $\tilde u$ is again a curve with boundary on $L$ (rather than with a moving Lagrangian boundary condition), and we wish to consider a different orientation scheme for $\tilde u$, namely the usual scheme for orienting moduli spaces of curves with boundary on $L$ using the given spin structure.
We shall be using this orientation scheme from now on, and we observe that it uses the trivialisation of $TL|_{\bd \tilde u}$ coming from the spin structure, which may be different from the pushforward trivialisation of $TL|_{\bd \tilde u}$ mentioned before. We denote the sign difference between the two orientations for $\tilde u$ by $(-1)^{\epsilon}=\pm 1$; this sign equals $+1$ is and only if the two trivialisations of $TL|_{\bd\tilde u}$ from above are homotopic. (Ultimately, we are going to use the fact that for our choice of orientation scheme, the moduli space of constant unconstrained disks is positively oriented; this may not be true for the $\gamma$ push-forward orientation scheme.) 

\begin{remark}
	\label{rem:signs_contract}
In general, the number $\epsilon$ does not necessarily equal $\epsilon(l)$ from the introduction. This equality holds when $\bd \tilde u$ is contractible, as easily seen from the definitions.
\end{remark}

Fourth, near the interior puncture $\tilde u$ satisfies the Hamiltonian Floer equation with original datum $\{F_t,J_t\}$ at the interior puncture, and is asymptotic orbit $y$. So we can glue the $y$-orbits, passing to Figure~\ref{fig:loop_spread}(e), and Figure~\ref{fig:loop_spread}(f) is another drawing of the same domain we got after gluing: namely, the disk with $k+1$ boundary punctures and one interior unconstrained marked point, fixed at $0\in \C$. Let us explain the presence of the marked point $0\in\C$: it is carried over from a marked point on the upper disk in Figure~\ref{fig:loop_spread}(b), where the interior marked point serves to stabilise the domain; such a marked point is present in the definition of the Seidel element. 

Summing up,
\begin{equation}
\label{eq:CO}
\CO^k(\S(\gamma))(x_1\otimes \ldots\otimes x_k)=(-1)^{\epsilon}\cdot \rho(l)\cdot \sum \sharp \M^\gamma(x_1,\ldots,x_k;x_0)\cdot x_0 
\end{equation}
where $\M^\gamma(x_1,\ldots,x_k;x_0)$ is the 0-dimensional moduli space of  disks shown in Figure~\ref{fig:loop_spread}(f) which satisfy the inhomogeneous pseudo-holomorphic equation defined by domain- and modulus-dependent perturbation data in the sense of \cite{SeiBook08} such that:
\begin{itemize}
 \item the disks carry the unconstrained interior marked point fixed at $t=0$, the output boundary puncture fixed at $t_0=1$, and $k$ free input boundary punctures at $t_i\in S^1$, $i=1,\ldots,k$;
\item  on a strip-like end corresponding to a boundary puncture  $t_i\in S^1$,  perturbation data restrict to the Floer datum which is the  $\gamma_{t_i}$-pullback of the original Floer datum $\{H_s,J_s\}_{s\in [0,1]}$,
and the asymptotic chord for this strip must be the $\gamma_{t_i}$-pullback of the asymptotic chord $x_i$ of the original Floer datum;
\item
the data must be consistent with gluing strip-like ends
 at $t_i\in S^1$ to strip-like ends of punctured pseudo-holomorphic disks carrying the $\gamma_{t_i}$-pullbacks of the  perturbation data defining the \ai structure on $CF^*(L,L)$;
 back by $\gamma_{t_i}$. 

 \item we use the standard orientation scheme for curves with boundary on $L$ to orient the $\M^\gamma$s, and the sign $(-1)^\epsilon$ was explained above.
\end{itemize}

The counts $\sharp \M^\gamma$ are signed and weighted by $\rho$; the third condition guarantees that $\CO^*(\S(\gamma))$ is a Hochschild cocycle.
Formula (\ref{eq:CO}) 
coincides with  the formula from \cite[Section (10c)]{SeiBook08} defining the natural transformation $[\gamma^\sharp]$ up to $(-1)^{\epsilon}\rho(l)$, and we have clarified this difference.

\begin{remark}
 The fixed interior marked point at $t=0$ and the fixed boundary marked point at $t_0=1$ make sure our disks have no automorphisms, so the values $t_i\in S^1$ of the other boundary punctures are uniquely defined.
\end{remark}


Before proceeding, note that we are already able to compute $\CO^0(\S(\gamma))$.

\begin{corollary}
\label{cor:compute_CO_0}
If $\{\gamma_t\}_{t\in S^1}$ is a Hamiltonian loop such that $\gamma_t(L)=L$, then on the chain level, $$\CO^0(\S(\gamma))=(-1)^{\epsilon(l)}\cdot \rho(l)\cdot 1_L\in HF^0(L,L).$$
Here $1_L$ is a chain-level representative of the cohomology unit, see \eqref{eq:1_chain} below.
\end{corollary}
\begin{proof}
When $k=0$, the moduli space in formula~(\ref{eq:CO}) is exactly the moduli space defining the cohomological unit in $CF^*(L,L)$, see e.g.~\cite[Section~2.4]{She13}.  The equality $\epsilon=\epsilon(l)$ holds by Remark~\ref{rem:signs_contract}: for a small Hamiltonian, the curves computing the unit are close to being constant and therefore have contractible boundary.
\end{proof}

\begin{proof}[Proof of Theorem~\ref{th:CO_invt_lag}(a)]
This is the homology-level version of Corollary~\ref{cor:compute_CO_0}. 
\end{proof}

\subsection{The PSS maps in degree one}
Our goal will be to compute a ``topological piece'' of $\CO^1(S(\gamma))$. This subsection introduces some background  required for  the computation: in particular, we recall that there is a canonical map $\Phi\co H^1(L)\to HF^*(L,L)$ which was used in the statement of Theorem~\ref{th:CO_invt_lag}. This is  the Lagrangian PSS map of Albers \cite{Alb08},
and the fact it is canonical was discussed, for instance, by Biran and Cornea \cite[Proposition 4.5.1({\it ii})]{BC09A} in the context of Lagrangian quantum cohomology.

First, recall that once the Floer datum is fixed, the complex $CF^*(L,L)$ acquires the Morse $\Z$-grading. This grading is not preserved by the Floer differential or the \ai structure maps, but is still very useful.
Assume that the Hamiltonian perturbation, as part of the Floer datum, is chosen to have a unique minimum $x_0$ on $L$, which means that $CF^0(L,L)$ is one-dimensional and generated by $x_0$. We denote by 
\begin{equation}
\label{eq:1_chain}
 1_L\in CF^0(L,L)
 \end{equation}
the chain-level cohomological unit defined in \cite[Section~2.4]{She13}, which is proportional to $x_0$. Now pick a metric  and a Morse-Smale function $f$ on $L$ with a single minimum; together they define the Morse complex which we denote by $C^*(L)$.  
Consider the ``Maslov index 0'' versions of the PSS maps, which are linear (but \emph{not} chain) maps:
\begin{equation}
\label{eq:PSS_chain}
\Psi\co CF^*(L,L)\to C^*(L),\quad \Phi\co C^*(L)\to CF^*(L,L). 
\end{equation}
These maps are defined as in the paper of Albers \cite{Alb08}, with the difference that $\Phi,\Psi$ count configurations with Maslov index 0 disks only.
For example, the map $\Psi$ counts configurations consisting of a Maslov index 0 pseudo-holomorphic disk with boundary on $L$ and one input boundary puncture, followed by a semi-infinite gradient trajectory of $f$ which outputs an element of $C^*(L)$. Similarly, $\Phi$ counts  configurations in which a semi-infinite gradient trajectory is followed by a Maslov index 0 disk with an output boundary puncture. The maps $\Psi,\Phi$ preserve $\Z$-gradings on the two complexes.

Let $d_0\co CF^*(L,L)\to CF^{*+1}(L,L)$ be the ``Morse'' part of the Floer differential counting the contribution of Maslov 0 disks, see Oh \cite{Oh96A}. Denote by $d_{\it{Morse}}\co C^*(L)\to C^{*+1}(L)$ the usual Morse differential. The lemma below is a version of \cite[Theorem~4.11]{Alb08}.

\begin{lemma}
\label{lem:PSS_general}
 $\Phi,\Psi$ are chain maps with respect to $d_0$ and $d_{\it{Morse}}$, and are cohomology inverses of each other.\qed
\end{lemma}


\begin{lemma}
	\label{lem:phi_cocycle}
Suppose $HF^*(L,L)\neq 0$.
If $y\in C^1(L)$ is a Morse cocycle (resp.~coboundary) then $\Phi(y)$ is a Floer cocycle (resp.~coboundary).
\end{lemma}

\begin{proof}
 This follows from the fact that the image and the kernel of $d_0$ and the full Floer differential $d$ coincide on $CF^1(L,L)$ if $HF^*(L,L)\neq 0$, by Oh's decomposition of the Floer differential \cite{Oh96A}.
\end{proof}
Consequently, if $HF^*(L,L)\neq 0$, we get a map
$$
\Phi\co H^1(L)\to HF^*(L,L).
$$
For $\Psi$, we have a weaker lemma using \cite{Oh96A} (this lemma is not true for coboundaries instead of cocycles).
\begin{lemma}
	\label{lem:psi_cocycle}
 If $y\in CF^*(L,L)$ is a Floer cocycle, then $\Psi(y)\in C^*(L)$ is a Morse cocycle.\qed
\end{lemma}

By Lemmas~\ref{lem:phi_cocycle},~\ref{lem:psi_cocycle}, given $HF^*(L,L)\neq 0$, we have the following induced maps which we denote by the same symbols $\Psi,\Phi$, abusing notation:
\begin{equation}
\label{eq:PSS_homol}
\Psi\co CF^*(L,L)\to H^*(L),\quad \Phi\co H^*(L)\to HF^*(L,L). 
\end{equation}
In particular, Theorem~\ref{th:CO_invt_lag} in Section~\ref{sec:intro} refers to this cohomological version of the map $\Phi$.
We remind that $\Psi$ does not necessarily descend to a map from $HF^*(L,L)$.

\subsection{Computing the topological part of $\CO^1(\S(\gamma))$}
We continue to use the above conventions and definitions, namely we use the $\Z$-grading on $CF^*(L,L)$,  the maps $\Phi,\Psi$, and the choice of a Hamiltonian perturbation on $L$ with a unique minimum $x_0$. From now on, we assume $HF^*(L,L)\neq 0$.
Recall that $\CO^*(\S(\gamma))$ is determined via formula (\ref{eq:CO}) by the moduli spaces $\M^\gamma(x_1,\ldots,x_k;x_0)$. The connected components of $\M^\gamma(x_1,\ldots,x_k;x_0)$ corresponding to disks of Maslov index $\mu$
have dimension $$|x_0|+k+\mu-\sum_{i=1}^k |x_i|$$ where $|x_i|$ are the $\Z$-gradings of the $x_i\in CF^*(L,L)$. Consequently,
$$\CO^k(\S(\gamma))\co  CF^*(L,L)^{\otimes k}\to CF^*(L,L)$$ is a sum of maps of degrees 
$$-k-mN_L,\quad m\ge 0,$$
where $N_L$ is the minimal Maslov number of $L$. In particular, the  restriction of $\CO^1(\S(\gamma))$ to ${CF^1(L,L)}$
is of pure degree $-1$, that is, its image lands in $CF^0(L,L)$:
$$\CO^1(\S(\gamma))|_{CF^1(L,L)}\co CF^1(L,L)\to CF^0(L,L).$$
Moreover, this map is determined by the moduli space consisting of Maslov index~0 disks only, and can be computed in purely topological terms. This is the main technical computation which we now perform; recall that $l\in H_1(L)$ is the homology class of an orbit of $\gamma$ in $L$, and the sign $(-1)^{\epsilon(l)}$ was defined in Section~\ref{sec:intro}.

\begin{proposition}
\label{prop:compute_CO_1}
 Suppose $HF^*(L,L)\neq 0$. If $x\in CF^1(L,L)$ is a Floer cocycle, then on the chain level,
 $$
\CO^1(\S(\gamma))(x)=(-1)^{\epsilon(l)}\cdot \rho(l)\cdot  \langle\Psi(x),l\rangle\cdot 1_L.
$$
Here $\Psi(x)\in H^1(L)$ is from~(\ref{eq:PSS_homol}), $\langle-,-\rangle$ denotes the pairing $H^1(L)\otimes H_1(L)\to \k$, we consider $\CO^1$ on the chain level, and $1_L$ is a chain-level cohomology unit as in~\eqref{eq:1_chain}.
\end{proposition}

\begin{proof}
All disks with boundary on $L$ we consider in this proof are assumed to have Maslov index 0.
We identify the domains of all disks that appear in the proof with the unit disk in $\C$, and their boundaries are identified with the unit circle $S^1\subset \C$.
In the subsequent figures, punctured marked points will be drawn by circles filled white, and unpunctured marked points by circles filled black.
According to formula~(\ref{eq:CO}), for a generator $x\in CF^1(L,L)$ we have 
$$\CO^1(\S(\gamma))(x)=(-1)^\epsilon\cdot \rho(l)\cdot \sharp \M^\gamma(x;x_0),$$
where  $\M^\gamma(x;x_0)$ consists of (perturbed pseudo-holomorphic) Maslov index 0 disks whose domains are shown in Figure~\ref{fig:Two_Gluings}(a).

\subsubsection*{Step 1. Perturbation data producing bubbles with unpunctured points}
Recall that the domains appearing in the moduli space $\M^\gamma(x;x_0)$
are disks with the interior marked point $0$ and boundary punctures $1,t$, where $t\in S^1\setminus\{1\}$.
For further use, we will choose perturbation data defining $\M^\gamma(x;x_0)$ whose bubbling behaviour as $t\to 1$ differs from the standard one. Usually, the perturbation data would be chosen so as to be compatible, as $t\to 1$, with the gluing shown in Figure~\ref{fig:Two_Gluings}(a)$\to$(b), where the bubble meets the principal disk 
along a puncture, meaning that near this puncture it satisfies a  Floer equation and shares an asymptotic Hamiltonian chord with the corresponding puncture of the principal disk.
On the other hand, we will use perturbation data consistent with gluing shown in Figure~\ref{fig:Two_Gluings}(a)$\to$(c), where the bubble is attached to the principal disk by an unpunctured marked point. Near the unpunctured marked point, the disks satisfy a holomorphic equation with no Hamiltonian term.

\begin{figure}[h]
 \includegraphics{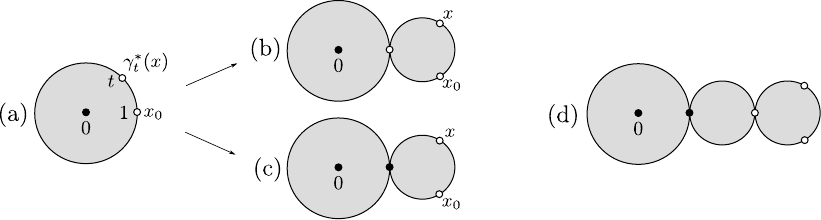}
\caption{Two types of gluings for $\M^\gamma(x;x_0)$, and a way to interpolate between the glued perturbation data.}
\label{fig:Two_Gluings}
\end{figure}

Let us explain how to define both types of data more explicitly. The domain in Figure~\ref{fig:Two_Gluings}(a), with free parameter $t$ close to $1$, is bi-holomorphic to the domain shown in Figure~\ref{fig:Disk_stretch} whose boundary marked points are fixed at $1$ and some $t_0\in S^1$, upon which a stretching procedure along the strip labelled (b) is performed. This stretching procedure changes the complex structure on the disk by identifying the strip with $[0,1]\times [0,1]$, removing it, and gluing back the longer strip $[0,1]\times [0,r]$. The parameter $r\in [1,+\infty)$ is free and replaces the free parameter $t$, so that tending $r\to+\infty$ replaces  the collision of two marked points $t\to 1$.

\begin{figure}[h]
 \includegraphics{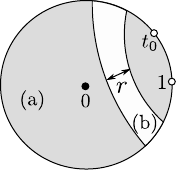}
\caption{Collision of two boundary marked points seen as stretching the strip (b) with parameter $r\to+\infty$.}
\label{fig:Disk_stretch}
\end{figure}

In order to get perturbation data which are consistent with the usual bubbling, shown in Figure~\ref{fig:Two_Gluings}(a)$\to$(b), one requires  the perturbed pseudo-holomorphic equation to coincide, on the strip $[0,1]\times [0,r]$, with the usual Floer equation defining the Floer differential, which uses a Hamiltonian perturbation translation-invariant in the direction of $[0,r]$. In order to get perturbation data producing the bubbling pattern Figure~\ref{fig:Two_Gluings}(a)$\to$(c), we simply put an unperturbed pseudo-holomorphic equation on the strip $[0,1]\times [0,r]$, without using a Hamiltonian perturbation at all.

Both ways of defining perturbation data are subject to appropriate gluing and compactness theorems, which precisely say that as we tend $r\to\infty$, the solutions bubble in one of the two corresponding ways shown in Figure~\ref{fig:Two_Gluings}. 
The standard choice is used, for example, to prove that  $[\gamma^\sharp]$  (obtained from the counts of various $\M^\gamma$) is a Hochschild cocycle in the Fukaya \ai algebra of $L$ defined using Seidel's setup with Hamiltonian perturbations. The other choice will be more convenient for our computations. 
Note that the two different types of perturbation data give the same count $\sharp \M^\gamma(x;x_0)$: this is proved by interpolating between them using the two-parametric space of perturbation data obtained from  gluing together the disks in Figure~\ref{fig:Two_Gluings}(d) with different length parameters. Recall that all disks in $\M^\gamma(x;x_0)$ have Maslov index 0 as the Morse index $|x|=1$, therefore no unnecessary bubbling occurs. (Since we do not want to compute the moduli spaces $\M^\gamma$ other than the $\M^\gamma(x;x_0)$ for $|x|=1$, we do not have to worry about extending our unusual type of perturbation data to the other moduli spaces.)

In addition, we will assume that the Hamiltonian perturbation vanishes over the principal disk in Figure~\ref{fig:Two_Gluings}(c), making this disk $J$-holomorphic and hence constant, because the disk has Maslov index 0. Such configurations can be made consistent with gluing: for this, one just needs to make the Hamiltonian perturbation vanish over  subdomain (a) in Figure~\ref{fig:Disk_stretch}, for all $t$ close to $1$.
Note that regularity can be achieved by perturbing the pseudo-holomorphic equation over the subdomain to the right of the strip (b) in Figure~\ref{fig:Disk_stretch}.

\subsubsection*{Step 2. A one-dimensional cobordism from $\M^\gamma(x;x_0)$}
In what follows, we will use (and assume familiarity with) the theory of holomorphic pearly trees developed by Sheridan in his Morse-Bott definition of the Fukaya category \cite{She13}; see also the earlier work of Cornea and Lalonde \cite{CL06}. Sheridan performs the analysis based on extending Seidel's setup of Fukaya categories from \cite{SeiBook08}. Although \cite{She13} considers exact Lagrangians instead of monotone ones, all the analysis works equally well in the non-exact case if we only consider  disks of  Maslov index 0, because here unpunctured disk bubbles cannot occur just like in the exact case. Techniques for dealing with holomorphic pearly trees  (or ``clusters'') with disks of arbitrary Maslov index have appeared in~\cite{CL06,Cha12}, but we will not actually need to appeal to them.

\begin{figure}[h]
\centerline{ \includegraphics{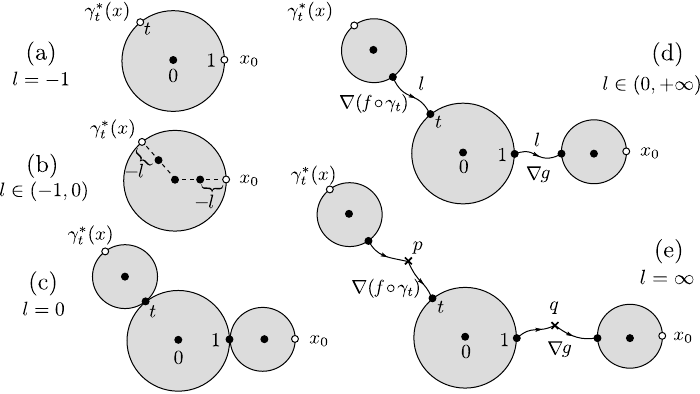}}
\caption{The domains for $s\in (0,2\pi)$, $l\in [-1,\infty]$, where $t=e^{is}$.}
\label{fig:loop_morse}
\end{figure}

We will now define a family of domains depending on two parameters $s\in [0,2\pi]$, $l\in [-1,+\infty]$.
When $s\notin \{0,2\pi\}$, the domains are shown in Figure~\ref{fig:loop_morse}(a)--(e), where we denoted $t=e^{is}$; we discuss the case $s\in \{0,2\pi\}$ later.
 When $l=-1$ the domain is the disk from the definition  of $\M^\gamma(x;x_0)$. When $l\in (-1,0)$, the domain is the same disk (called principal) with two additional interior marked points whose position is determined by the parameter $l$: the first point lies on the line segment $[0,t]$, the second one lies on the line segment $[0,1]$, and both points have distance $1+l$ from $0$. When $l=0$, the domain consists of the principal disk with marked points $0,1,t$, and two bubble disks attached to the principal disk at points $1$ and $t$. The first bubble disk has marked points $0,1$ and a boundary puncture at $-1$, the second one has marked points at $0,-1$ and a boundary puncture at $1$. When $0<l<\infty$, the domain contains the same three disks, now disjoint from each other, plus two line segments of length $l$ connecting the bubble disks to the principal one along the boundary marked points at which the disks used to be attached to each other. (So far, the length is just a formal parameter associated with the domain, but soon it will become the length of the flowline corresponding to the segment.) When $l=\infty$, we replace each line segment by two rays $[0,+\infty)\sqcup (-\infty,0]$. 

When $s=0$ or $s=2\pi$, the domains obtain extra bubbles as those discussed above in the definition of $\M^\gamma(x;x_0)$, which correspond to the parameter $t=e^{is}\in S^1$ approaching $1\in S^1$ from the two sides. These domains are shown in Figure~\ref{fig:Circle_Bubbles}: as $l$ goes from $-1$ to $0$, the two interior marked points move along the punctured paths.  Observe that these points are crossing the node between the two disks at some intermediate value of $l$; this does not cause any difficulty with the definitions because these marked points are only used to represent varying perturbation data consistent with the types of bubbling  we prescribe in the figures. We will soon mention what these varying data are in terms of stretching certain strips inside a fixed disk. When $l>0$, the length of the paths equals $l$. When $l=\infty$, one introduces broken lines $[0,+\infty)\sqcup (-\infty,0]$ as above.

\begin{figure}[h]
 \includegraphics{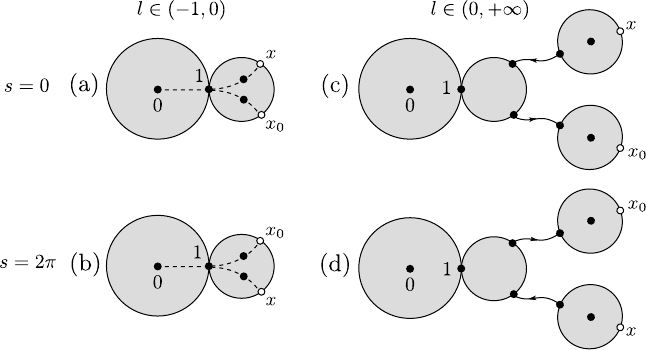}
\caption{The domains for $s\in \{0,2\pi\}$, $l\in [-1,\infty]$.}
\label{fig:Circle_Bubbles}
\end{figure}

Having specified the domains, we  briefly  explain how to equip the disks with suitable perturbed pseudo-holomorphic equations, and line segments with suitable gradient equations to get a moduli space of solutions.
When $l=-1$, we choose the equations defining $\M^\gamma(x;x_0)$ as discussed above in Step~1, in particular which is consistent with bubbling at the unpunctured point as $s\to 0$ or $s\to 2\pi$. When $-1<l<0$, we choose the equations with the same properties as for $\M^\gamma(x;x_0)$, which are additionally constistent with bubbling at unpunctured points as $l\to 0$. When $l\ge 0$, we choose the equation on the disk with an input puncture to be the $\gamma_t$-pullback of the one appearing in the definition of the PSS map $\Psi$, and the equation on the disk with an output puncture to be exactly the equation from the PSS map $\Phi$. 
Finally we fix two generic Morse-Smale functions $f,g\co L\to \R$. On the line segments and rays, the equation is the gradient equation for $g$ or the $\gamma_t$-pullback of the gradient equation for $f$, as shown in Figure~\ref{fig:loop_morse}. 

In general, we can arrange the equations on all disks to have a Hamiltonian perturbation. For the disks with punctures, we must do so anyway; however, for the disks without punctures (with unpunctured marked points only), we can choose the equation to be $J$-holomorphic 
without a Hamiltonian term, provided that we check that our moduli space can be made regular with this restricted choice. In order to carry out the computations below, we choose the zero Hamiltonian perturbation on all disks without boundary punctures, namely:
\begin{itemize}
	\item the central disk in Figures~\ref{fig:loop_morse}(c),~(d),~(e), 
	\item the left disks in Figures~\ref{fig:Circle_Bubbles}(a)--(d) and the central disks in Figures~\ref{fig:Circle_Bubbles}(c),~(d).
\end{itemize}

We will now specify one more property of the equations that we choose.
When $s=0$ and $l<0$, we require the perturbation data on the twice-punctured disk in Figure~\ref{fig:Circle_Bubbles}(a) to be obtained by $\pi$-rotation from the perturbation data on the similar disk for $s=2\pi$ (and the same parameter $l$), if we identify the two punctures with points $1,-1$ of the unit disk. When $s=0$ and $l\ge 0$, we make a similar symmetric choice.

Finally, we specify the asymptotic conditions at the punctures. If $x$ is a generator of $CF^1(L,L)$, we specify that the input puncture in Figures~\ref{fig:loop_morse} must be asymptotic to the $\gamma_t$-pullback of $x$ (as usual, if $x$ is a linear combination of generators, we take the disjoint union of the relevant moduli spaces). The output puncture must be asymptotic to the unique generator $x_0\in CF^0(L,L)$. When $l=\infty$, the first pair of rays in Figure~\ref{fig:loop_morse}(e) must be asymptotic to a point $p$ such that $\gamma_t^{-1}(p)\in C^1(L)$ (that is, $p$ is an index 1 critical point of $f\circ \gamma_t$) and the second pair of rays must be asymptotic to $q\in C^0(L)$; we assume $q$ is the unique minimum of $g$. The interior marked points on the disks are unconstrained.

Above, we have specified a 2-dimensional space of domains and the equations over them. This gives us a moduli space of solutions (``pearly trees'') which is 1-dimensional, by our choice of indices. We remind the reader that a formal definition of this moduli space falls into the setup of moduli spaces of pearly trajectories given by Sheridan \cite{She11}. For our purposes, its description given above will suffice.

The boundary of our 1-dimensional moduli space consists of:
\begin{itemize}
 \item solutions whose domains have parameter $l=-1$ or $l=\infty$,
\item solutions whose domains have parameter $s=0$ or $s=2\pi$.
\end{itemize}
We claim that solutions of the second type cancel pairwise. Indeed, recall that the disks without boundary punctures in Figures~\ref{fig:Circle_Bubbles}(a)--(d) are constant, and the perturbation data on the punctured disks for $s=0,2\pi$ are chosen in a way to provide the same solutions, after a $\pi$-rotation on each disk. Let us describe more explicitly what happens when $l<0$, as the case when $l\ge 0$ is clear enough from Figures~\ref{fig:Circle_Bubbles}(c),~(d). We can represent the domains shown in Figure~\ref{fig:loop_morse}(a), with free $l<0$ and free small $s>0$, where $t=e^{is}$, by a disk with fixed boundary punctures, stretched with length parameters $-1/l$ and $1/s$ along the three strips shown in Figure~\ref{fig:two_stretches_bubble}(a). The stretching procedure was described earlier, and our choice of perturbation data says that the stretched strips, and the sub-domain to the left of the $1/s$-strip, carry an unperturbed pseudo-holomorphic equation. So for $s=0$ we get the disks shown in Figure~\ref{fig:two_stretches_bubble}(b), with the unpunctured boundary marked point attached to a constant disk, which means this boundary marked point unconstrained. (As usual, the domain is considered up to complex automorphisms, so the unconstrained point does not prevent us from having rigid solutions.) This way, Figures~\ref{fig:Circle_Bubbles}(a)~and~\ref{fig:two_stretches_bubble}(b) are drawings of the same configuration, for any $l<0$. If we rotate the disk in Figure~\ref{fig:two_stretches_bubble}(b) by $\pi$, we get precisely the disk with perturbation data  we would have got for $s=2\pi$, except that the boundary marked point (with the attached constant disk) is on the different side of the boundary.

\begin{figure}[h]
	\centerline{\includegraphics{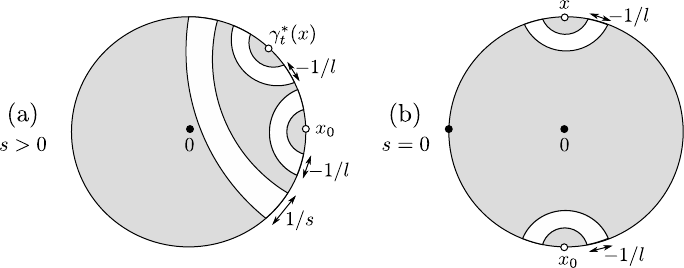}}
	\caption{Left: the domains for $l<0$, $s>0$ seen as a fixed disk with three stretched strips. Right: the same domains for $s=0$ when the principal disk is constant.}
	\label{fig:two_stretches_bubble}
\end{figure}

We have shown that the solutions for $s=0$ and $s=2\pi$ are in a natural bijection; we claim that this bijection reverses the signs associated to those solutions as parts of the moduli space. This is ultimately related to the fact that the constant disk is being glued to those solutions at the opposite boundary components. One can adopt a proof (which we will not provide in detail here) from the following classical example where an analogous sign issue has been treated. Suppose $L_0,L_1$ are monotone Lagrangian submanifolds with obstruction numbers $m_0(L_i)\in\Z$, $i=0,1$. Then the Floer differential $d$ on $CF^*(L_0,L_1)$ satisfies: $d^2=m_0(L_0)-m_0(L_1)$. This relation arises from Maslov index~2 disks bubbling off the two sides of a 1-dimensional moduli space of Floer strips \cite{Oh95}, see also \cite[(2.3.9)]{She13} and \cite[Figure~2]{Smith15}. Here, indeed, the gluing of the side bubbles contributes with the opposite signs and results in the term $m_0(L_0)-m_0(L_1)$. The reference for the signs in this relation is \cite{FO3Book}, see specifically Remark~3.2.21(1), Formula~(3.3.4) and Chapter~8 from that book.
Although in our case we would be gluing a constant disk rather than Maslov index~2 disk, and the non-constant curve in Figure~\ref{fig:two_stretches_bubble}(b) satisfies a different equation than the standard Floer one (e.g.~our equation is not $\R$-invariant), the required orientation analysis is essentially the same.

The outcome of the cancellation discussed above is that the count of configurations in Figure~\ref{fig:loop_morse}(a), i.e.~$\sharp\M^\gamma(x;x_0)$, equals  the count of configurations in Figure~\ref{fig:loop_morse}(e), and it remains to compute the latter.

\subsubsection*{Step 3. A Morse-theoretic computation}
Let us look at Figure~\ref{fig:loop_morse}(e). Recall that  $q\in L$ is the minimum of $g$, so the semi-infinite flowline of $\nabla g$ flowing into $q$ must be constant. Second, we have arranged the principal (central) disk to be constant, as well. So the configurations in Figure~\ref{fig:loop_morse}(e) reduce to those shown in Figure~\ref{fig:loop_count}.

\begin{figure}[h]
 \centerline{\includegraphics{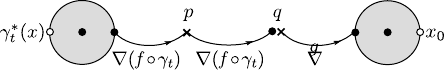}}
\caption{The domains when $l=+\infty$ and the principal disk together with a flowline are constant. Here $p$ is an index 1 critical point of $f\circ \gamma_t$, and $q$ is the minimum of $f$.}
\label{fig:loop_count}
\end{figure}

The free parameter $t=e^{is}\in S^1\setminus \{1\}$ is ``unseen'' by the domain after the principal disk became ghost (i.e.~constant), but the equations still depend on it. First, consider the left disk and the left flowline in Figure~\ref{fig:loop_count}, forgetting the rest of the configuration. Those disk and flowline satisfy the $\gamma_t$-pullback of the equation defining the PSS map $\Psi$, so for each $t$ the linear combination of points $p$ appearing as limits of such configurations equals $\gamma_t(\Psi(x))$, where $\Psi(x)\in C^1(L)$ is the PSS image which is a linear combination of index 1 critical points of $f$, so that $\gamma_t(\Psi(x))$ is a combination of critical points of $f\circ\gamma_t$. 

Let us now add back the middle flowline, still forgetting the right flowline and the right disk, and count the resulting configurations. The middle flowline is a semi-infinite flowline of $\nabla(f\circ \gamma_t)$ ending at the point $q$; note that $q$ is \emph{not} a critical point of $f$. 
Suppose for the moment that we allow the right end of the middle flowline to be free (not constrained to $q$) and denote the moduli space of such configurations by $P$. Then we can consider the evaluation map at the right end of the flowline, $\ev\co P\to L$. The image of $\ev$, as a chain, is a linear combination of unstable manifolds, with respect to the function $f\circ \gamma_t$, associated with the linear combination of the critical points $p$ which we have previously computed. Recall that this linear combination of points $p$ equals $\gamma_t(\Psi(x))$. Consequently, if we denote by $C_{\Psi(x)}\subset L$ the disjoint union of (oriented, codimension 1) unstable manifolds of the Morse cochain $\Psi(x)\in C^1(L)$ with respect to $f$, then 
$$
P=(S^1\setminus\{1\})\times C_{\Psi(x)},\qquad \ev(t,z)=\gamma_t(z).
$$
Those configurations which evaluate at $q\in L$ are  the intersection points $C_{\Psi(x)}\cap l$, where $l=\{\gamma_t(q)\}_{t\in S^1}$ is the orbit of $q$. By perturbing $\gamma_t$ and $f$, the intersections can be easily made transverse, and we get:
$$
\sharp (P\times_\ev \{q\})=[C_{\Psi(x)}]\cdot [l]=\langle \Psi(x),l\rangle.
$$
Recall this is the count of the part of confugurations in Figure~\ref{fig:loop_count} which end up at $q$.
Finally, the count of the rightmost flowlines (emerging from $q$) plus the right disks in Figure~\ref{fig:loop_count} equals $1_L\in CF^0(L,L)$. Indeed, the $g$-unstable manifold of the minimum $q$ is the whole manifold $L$ (minus a codimension 2 subset), so the count is the same as the count of the rightmost disks only, and the latter by definition produces $1_L$.

Putting everything together and noting that $\epsilon=\epsilon(l)$ by Remark~\ref{rem:signs_contract} (since the total boundary in Figure~\ref{fig:loop_count} is contractible, for a small Hamiltonian perturbation), we get the statement of Proposition~\ref{prop:compute_CO_1}. 
One last thing is to argue that the moduli spaces we have been using were regular. 

According to \cite{She13}, the regularity of moduli spaces of pearly trajectories consisting of pseudo-holomorphic disks and flowlines is equivalent to the regularity of the separate disks and flowlines not constrained to satisfy the incidence conditions, plus the transversality of the evaluation maps which account for the incidence conditions.

The non-constant disks in the proof carry the pseudo-holomorphic equation with a Hamiltonian perturbation  which makes them regular.
The constant disks  are  known to be regular on their own; and it is easy to see that for generic $f$, $g$, the flowlines are transverse to the evaluation maps for all appearing configurations.
\end{proof}

\subsection{Checking non-triviality in Hochschild cohomology}
In this subsection we prove Theorem~\ref{th:CO_invt_lag}(b),~(c) (recall that part (a) was proved earlier, see Corollary~\ref{cor:compute_CO_0}). We have computed in Proposition~\ref{prop:compute_CO_1} the map $\CO^1(\S(\gamma))|_{CF^1(L,L)}$, and it remains to see when the result survives to something non-trivial on the level of Hochschild cohomology, and thus  distinguishes $\CO^*(\S(\gamma))\in HH^*(L,L)$ from the unit in $HH^*(L,L)$.

First, let us quickly recall the definition of Hochschild cohomology. Let $A$ be an \ai algebra, and assume it is $\Z/2$ graded if $\chr \k\neq 2$. When $\chr \k=2$, we consider $A$ as an ungraded algebra. The space of Hochschild cochains is, by definition $$CC^*(A,A)=\prod_{k\ge 0}Hom(A^{\otimes k},A).$$ If $A$ is $\Z/2$-graded then $CC^*(A,A)$ is $\Z/2$-graded: $CC^r(A,A)=\prod_{k\ge 0}Hom(A^{\otimes k},A[r-k])$.
If $h=\{h^k\}_{k\ge 0}\in CC^*(A)$, $h^k\co A^{\otimes k}\to A$, then the Hochschild differential of $h$ is the sequence of maps
$$
\begin{array}{l}
(\bd h)^k(a_k,\ldots, a_1)=\\
\sum_{i+j\le k}(-1)^{(r+1)(|a_1|+\ldots+|a_i|+i)}\cdot\mu^{k+1-i}(a_k,\ldots a_{i+j+1},h^j(a_{i+j},\ldots,a_{i+1}),a_i,\ldots, a_1)+\\
\sum_{i+j\le k}(-1)^{r+1+|a_1|+\ldots+|a_i|+i}\cdot h^{k+1-i}(a_k,\ldots a_{i+j+1},\mu^j(a_{i+j},\ldots,a_{i+1}),a_i,\ldots, a_1).
\end{array}
$$
Here $r$ is the $\Z/2$-degree of $h$.  (When $k=0$, the agreement is that $Hom(A^{\otimes 0},A)=A$, so $h^0$ is an element of $A$.) If $\chr \k=2$, we do not need the gradings as the signs do not matter.

Let us return to the \ai algebra $CF^*(L,L)$. We continue to use the Morse $\Z$-grading on the vector space $CF^*(L,L)$ keeping in mind this grading is not respected by the \ai structure. If $L$ is oriented, the reduced $\Z/2$-grading is preserved by the \ai structure so $CF^*(L,L)$ is a $\Z/2$-graded \ai algebra. If $L$ is not oriented, we must suppose $\chr \k=2$.

\begin{proof}[Proof of Theorem~\ref{th:CO_invt_lag}(b)]

We continue to work with $\CO^*$ on chain level. Because the homological closed-map is unital \cite[Lemma~2.3]{She13},
 the Hochschild cohomology unit  is realised by the cochain $1_{HH}\coloneqq \CO^*(1)\in CC^*(L,L)$, where $1$ is the unit in $QH^*(X)$. The \ai category $CF^*(L,L)$ need not be strictly unital, so the maps $(1_{HH})^k=\CO^k(1)$ need not vanish for $k>0$. However, because the identity Hamiltonian loop preserves $L$ and has homologically trivial orbits on it, Proposition~\ref{prop:compute_CO_1} applies to $1=\S(\id)\in QH^*(X)$ and says that $(1_{HH})^1(x)=0$ for any Floer cocycle $x\in CF^1(L,L)$.

Suppose $\CO^*(S(\gamma))+\alpha \cdot 1_{HH}$ is the coboundary of an element $h\in CC^*(L,L)$, for some $\alpha\in \k$. By comparing $(\bd h)^0$ with $\CO^0(S(\gamma))+\alpha\cdot (1_{HH})^0$, see Corollary~\ref{cor:compute_CO_0}, we get:
$$
 \mu^1(h^0)=(-1)^{\epsilon(l)}\rho(l)\cdot 1_L+\alpha\cdot 1_L.
$$
Here $\mu^1$ is the Floer differential and $1_L$ is a chain-level cohomology unit \eqref{eq:1_chain}. The assumption  $HF^*(L,L)\neq 0$ implies that the Floer cohomology unit $1_L$ cannot be killed by the Floer differential. Therefore, we cannot solve the above equation unless  $\alpha=-(-1)^{\epsilon(l)}\rho(l)$ and $\mu^1(h^0)=0$.
Next, by comparing 
$(\bd h)^1$ with $\CO^1(S(\gamma))+\alpha\cdot (1_{HH})^1$, see Proposition~\ref{prop:compute_CO_1}, for any Floer cocycle $x\in CF^1(L,L)$ we get
\begin{multline*}
(-1)^{|h^0|+1}\mu^2(x,h^0)+(-1)^{(|h^0|+1)(|x|+1)}\mu^2(h^0,x)
\\+
\mu^1(h^1(x))+(-1)^{|h^1|+1}h^1(\mu^1(x))
=(-1)^{\epsilon(l)}\rho(l)\cdot \langle\Psi(x),l\rangle\cdot 1_L.
\end{multline*}
Here we are using the version of $\Psi$ as in (\ref{eq:PSS_homol}).
Because $x$ is a Floer cocycle, the last summand of the left hand side vanishes. If $\chr \k=2$, redenote $a\coloneqq h^0\in CF^*(L,L)$. If $\chr \k\neq 2$, let $a\in CF^{odd}(L,L)$ be the odd degree part of $h^0$. By computing the signs in the above equality we get, for any Floer cocycle $x\in CF^1(L,L)$:
$$
\mu^2(x,a)+\mu^2(a,x)+
\mu^1(h^1(x))=(-1)^{\epsilon(l)}\rho(l)\cdot \langle\Psi(x),l\rangle\cdot 1_L.
$$
Recall that $\mu^1(h^0)=0$ so $\mu^1(a)=0$ as well, and we get the following equality for Floer cohomology classes $[x],[a]\in HF^*(L,L)$ and $\Psi(x)\in H^1(L)$:
$$
\mu^2([x],[a])+\mu^2([a],[x])=(-1)^{\epsilon(l)}\rho(l)\cdot \langle\Psi(x),l\rangle\cdot 1_L\in HF^*(L,L).
$$
Now put $x=\Phi(y)$, where $y\in C^1(L)$ is a Morse cochain and $\Phi$ is the chain-level map from (\ref{eq:PSS_chain}). The above equality means that for all $[y]\in H^1(L)$,
$$
\mu^2(\Phi([y]),[a])+\mu^2([a],\Phi([y]))=(-1)^{\epsilon(l)}\rho(l)\cdot \langle [y],l\rangle\cdot 1_L\in HF^*(L,L).
$$
This time, we have used the homology-level version of $\Phi$ from (\ref{eq:PSS_homol}).
The above equality is exactly prohibited by the hypothesis of Theorem~\ref{th:CO_invt_lag}(b), so  Theorem~\ref{th:CO_invt_lag}(b) is proved.
\end{proof}

\begin{proof}[Proof of Theorem~\ref{th:CO_invt_lag}(c)]
Note that, on the homology level, $$\CO^0(\S(\gamma)*Q)=\CO^0(\S(\gamma))\cdot \CO^0(Q)=1_L\cdot \CO^0(Q)=\CO^0(Q)$$ 
(here the dot denotes the $\mu^2$ product), so the only possible linear relation between $\CO^*(S(\gamma)*Q)$ and $\CO^*(Q)$ is that $$\CO^*((\S(\gamma)-1)*Q)=0,$$ 
where $1$ is the unit in $QH^*(X)$. We have $\CO^*((\S(\gamma)-1)*Q)=\CO^*(\S(\gamma)-1)\star \CO^*(Q)$, where the symbol $\star$ denotes the Yoneda product in Hochschild cohomology.

Let us now return to working with $\CO^*$ on the chain level. Recall that if $\phi=\{\phi^k\}_{k\ge 0}$, $\psi=\{\psi^k\}_{k\ge 0}\in CC^*(L,L)$ are Hochschild cochains, the $k=1$ part of their Yoneda product by definition equals
$$
(\phi\star\psi)^1(x)=\pm \mu^2(\phi^1(x),\psi^0)\pm\mu^2(\phi^0,\psi^1(x)).
$$
There is an explicit formula for the signs which we do not need. Let us apply this formula to $\CO^*(\S(\gamma)-1)$ and
$\CO^*(Q)$. We know that
 $(\CO^*(\S(\gamma)-1))^0=0$ by Corollary~\ref{cor:compute_CO_0}, and $(\CO^*(\S(\gamma)-1))^1(x)=\CO^1(\S(\gamma))(x)$ is given by Proposition~\ref{prop:compute_CO_1} for any Floer cocycle $x\in CF^1(L,L)$. Consequently, we get:
$$
\CO^1((\S(\gamma)-1)*Q)(x)=(-1)^{\epsilon(l)}\rho(l)\cdot\langle\Psi(x),l\rangle\cdot \CO^0(Q).
$$
From this point, the rest of the proof follows the one of Theorem~\ref{th:CO_invt_lag}(b).
\end{proof}

\section{The closed-open map for real toric Lagrangians}
\label{sec:real_toric}
In this section, after a short proof of Theorem~\ref{th:toric_CO_0}, we look for further examples of real toric Lagrangians where Theorem~\ref{th:CO_invt_lag} can be effectively applied. We also discover that Proposition~\ref{prop:compute_CO_1}, after additional work, allows to show that
the Fukaya \ai algebra of some of the considered Lagrangians is not formal.
In particular, we prove the results about real toric Lagrangians stated in Section~\ref{sec:intro} (except for Proposition~\ref{prop:CO_RPn_Inject} and Corollary~\ref{cor:RPn_generate}, which have been proved therein).
We work with a coefficient field $\k$ of characteristic two throughout this section.
\subsection{A proof of Theorem~\ref{th:toric_CO_0}}
Let $X$ be a compact, smooth toric Fano variety, and $D\subset X$ be a toric divisor corresponding to one of the facets of the polytope defining $X$. There is a Hamiltonian circle action $\gamma$ on $X$ associated with $D$, which comes from the toric action by choosing a Hamiltonian which achieves maximum on $D$. A theorem of McDuff and Tolman \cite{MDT06} says the following.

\begin{theorem}
\label{th:mcduff_tolman}
We have $\S(\gamma)=D^*$, where $D^*\in QH^*(X)$ is the Poincar\'e dual of~$D$. \qed
\end{theorem}

The loop $\gamma$ never preserves the real Lagrangian $L\subset X$, but if we parametrise $\gamma=\{\gamma_t\}_{t\in [0,1]}$ then $\gamma_{1/2}(L)=L$, see \cite{Ha13}. Consequently,  $\alpha=\{\gamma_t(L)\}_{t\in [0,1/2]}$ is a loop of Lagrangian submanifolds, and moreover we have $\alpha^2=\{\gamma_t(L)\}_{t\in [0,1]}$ in the space of Lagrangian loops.
There is an associated Lagrangian Seidel element $\S_L(\alpha)\in HF^*(L,L)$, which counts pseudo-holomorphic disks with rotating boundary condition $\alpha$, and a single boundary puncture which evaluates to an element of $HF^*(L,L)$.
A theorem of Hyvrier \cite[Theorem~1.13]{Hy16}, based on the disk doubling trick, computes $\S_L(\alpha)$.

\begin{theorem}
\label{th:hyvrier}
We have $\S_L(\alpha)=[L\cap D]^*$, where $L\cap D$ is the clean intersection that has codimension 1 in $L$, and $[L\cap D]^*\in H^1(L)\subset HF^*(L,L)$ is its dual class. \qed
\end{theorem}

The inclusion $H^1(L)\subset HF^*(L,L)$ is the PSS map $\Phi$ from Section~\ref{sec:main_proof}, which is injective because $HF^*(L,L)\cong H^*(L)$ by Theorem~\ref{th:haug}.

\begin{proof}[Proof of Theorem~\ref{th:toric_CO_0}]
It suffices to prove that $\CO^0(D^*)=\D(\F(D^*))$, where $D\subset X$ is a toric divisor as above and $D^*\in QH^*(X)$ is its dual class, because  such $D^*$ generate $QH^*(X)$ as an algebra \cite{MDT06}. Let $\gamma$ be the Hamiltonian loop corresponding to $D$ as above, and $\alpha$ be the Lagrangian loop as above, such that $\alpha^2=\{\gamma_t(L)\}_{t\in [0,1]}$.
It follows from Theorem~\ref{th:Charette_Cornea} that $$\CO^0(\S(\gamma))=\S_L(\alpha^2),$$ and the latter can be rewritten as
$\F(\S_L(\alpha))$, where $\F$ is the Frobenius map on $HF^*(L,L)$. 
By Theorem~\ref{th:mcduff_tolman}, $\S(\gamma)=D^*$, and by Theorem~\ref{th:hyvrier}, $\S_L(\alpha)=[L\cap D]^*$. Finally,
if we look at Haug's construction \cite{Ha13} of the Duistermaat isomorphism $\D$, we will see that $[L\cap D]^*=\D(D^*)$. Putting everything together, we get $$\CO^0(D^*)=\F(\D(D^*)).$$ Because $\D$ is a ring map, it commutes with the Frobenius maps on $HF^*(L,L)$ and $QH^*(X)$, and the theorem follows.
\end{proof}

\subsection{Split-generation for toric varieties with Picard rank 2}
It is known that the unique toric variety with Picard number $1$ is the projective space. 
By a  theorem of Kleinschmidt~\cite{Kl88}, see also~\cite{CMS10},  every  $n$-dimensional toric Fano variety whose Picard group has rank $2$ (i.e.~whose fan has $n+2$ generators) is isomorphic to the projectivisation of a sum of line bundles over $\CP^{n-k}$:
\begin{equation}
\label{eq:define_X}
X(a_1,\ldots,a_k)\coloneqq\P_{\CP^{n-k}}(\O\oplus\O(a_1)\oplus\ldots\oplus \O(a_k)),\ a_i\ge 0,\ \sum_{i=1}^ka_i\le n-k-1. 
\end{equation}
(The imposed conditions on the $a_i$ are equivalent to $X$ being toric Fano). The $n+2$ vectors in $\Z^n$ generating the fan of $X(a_1,\ldots,a_k)$ are the columns of the following matrix:
\begin{equation}
\label{eq:matrix_X}
\left(
\begin{array}{c@{}c@{}}
 \hspace{2em} I_{n\times n}\hspace{2em}
\begin{smallmatrix}
-1 &a_1\\
\vdots&\vdots\\
-1&a_k\\
0&-1\\
\vdots&\vdots \\
0&-1
\end{smallmatrix}
\end{array}
\right)
\end{equation}
The minimal Chern number of $X(a_1,\ldots,a_k)$ equals $\gcd(k+1,n-k+1-\sum a_i)$, see \cite{QR98}.
Some of these varieties provide further examples where, using Theorems~\ref{th:toric_CO_0}~\and~\ref{th:CO_invt_lag}, we can prove the injectivity of $\CO^*$ and  deduce split-generation. 

\begin{theorem}
\label{th:hirzebruch}
 Let $X\coloneqq X(a_1,\ldots,a_k)$ be as above, $L\subset X$ the real Lagrangian, $\k$ a field of characteristic 2.
Suppose all $a_i$ are odd and $\gcd(k+1,n-k+1-\sum a_i)\ge 2$.
\begin{enumerate}
 \item[(a)] If $n-k+1$ is odd, then $\CO^0\co QH^*(X)\to HF^*(L,L)$ is injective.
\item[(b)] If $n-k+1$ is even, $k$ is even and the numbers $a_i$ come in equal pairs, then $\CO^*\co QH^*(X)\to HH^*(L,L)$ is injective while $\CO^0$ is not.
\end{enumerate}
In both cases $L$ split-generates $\Fuk(X)_0$.
\end{theorem}

\begin{proof}
 Let $x,y\in H^{2}(X)$ be the generators corresponding to the last two columns of the matrix (\ref{eq:matrix_X}). They generate $QH^*(X)$ as an algebra and satisfy the following relations when $\chr\k=2$:
$$
 x(x+y)^{k}=1,\quad y^{n-k+1}(x+y)^{-\sum a_i}=1.
$$
(For brevity, we no longer use the symbol $*$ to denote the quantum product.)
If $n-k+1$ is odd, one can show that the Frobenius endomorphism $\F$ on $QH^*(X)$ is an isomorphism, so $\CO^0$ is injective by Theorem~\ref{th:toric_CO_0}. It follows that $\CO^*$ is also injective, and split-generation follows from Theorem~\ref{th:generation_crit}. Part (a) is proved.

In the rest of the proof we work with the case (b), so let us redenote:
$n-k+1=2r$, $k=2q$, $\sum a_i=2p$. The rewritten relations in $QH^*(X)$ are:
\begin{equation}
\label{eq:QH_X_p_q}
 x(x+y)^{2q}=1,\quad y^{2r}(x+y)^{-2p}=1.
\end{equation}

\begin{lemma}
\label{lem:hirzebruch_ker_f_generator}
For the ring $QH^*(X)$ as in (\ref{eq:QH_X_p_q}), $\ker \F$ is the ideal generated by $y^r(x+y)^{-p}+1$.
\end{lemma}

\begin{proof}
Equations (\ref{eq:QH_X_p_q}) are equivalent to
$$
x^{-p}=y^{2rq},\quad y^{4rq+2r}+y^{4rp+2p}+1=0,
$$
where the second equation is rewritten from the second equation in (\ref{eq:QH_X_p_q}) using the substitution $x^{-p}=y^{2rq}$. This means if we denote
$$
R(y)=y^{2rq+r}+y^{2rq+p}+1,
$$
then $R(y)=y^r(x+y)^{-p}+1$. Denote $g=\gcd(2rq,p)$ and let $\alpha,\beta\in \Z$ be such that
$$
-2rq\cdot\alpha+p\cdot \beta=g.
$$
Consider the map $\phi\co  \k[u]\to QH^*(X)$ given by $u\mapsto x^\alpha y^\beta$; this map is onto because  we get 
\begin{equation}
\label{eq:phi_acting_on_u}
\phi(u^{p/g})= y,\quad \phi(u^{2qr/g})= x^{-1} 
\end{equation}
using the given relations (note that the powers $p/g$, $2qr/g$ are integral). Further, $\ker\phi$ is obviously the ideal generated by $V(u)^2$ where $V(u)\coloneqq R(u^{p/g})$, and we conclude that $\phi$ provides an isomorphism 
\begin{equation}
\label{eq:QH_Hirzebruch_poly}
 \phi\co \k[u]/V(u)^2\xrightarrow{\cong}QH^*(X),\quad V(u)=u^{\frac p g (2rq+r)}+u^{\frac p g (2rq+p)}+1.
\end{equation}
It is clear that $V(u)$ generates the kernel of the Frobenius map on $\k[u]/V(u)^2$.
Because $V(u)$ corresponds to $y^r(x+y)^{-p}+1$ under $\phi$, Lemma~\ref{lem:hirzebruch_ker_f_generator} follows.
\end{proof}

We continue  the proof of Theorem~\ref{th:hirzebruch}(b). It turns out that, similarly to the case of $\RP^n\subset \CP^n$ studied in the introduction, 
the generator of $\ker\F$ from Lemma~\ref{lem:hirzebruch_ker_f_generator} equals $\S(\gamma)+1$
for a real Hamiltonian loop $\gamma$ on $X$ which preserves $L$ setwise and has homologically non-trivial orbits on it.
To construct $\gamma$, we will need the additional assumption that the $a_i$ come in equal pairs, so we assume the sequence $(a_i)_{i=1}^{2q}$ is $(a_1,a_1,\ldots,a_q,a_q)$.


 Recall that $X$, being a toric manifold, is a quotient of  $\C^{2r+2q+1}$ minus some linear subspaces determined by the fan, by an  action of $(\C^*)^2$. Using the common notation, this action is given by
$z\mapsto t_1^{v_1}t_2^{v_2}z$, where $z\in \C^{2r+2q+1}$ and $v_1,v_2$
are the vectors in $\Z^{2r+2q+1}$ given by the following two rows:
\begin{equation}
\label{eq:torus_weights}
\begin{array}{ccccc|ccc|cc}
\multicolumn{5}{c|}{
\begin{smallmatrix}2q \text{ \it entries}\end{smallmatrix}
}&
\multicolumn{3}{c|}{
\begin{smallmatrix}2r-1 \text{ \it entries}\end{smallmatrix}
}&
\multicolumn{2}{c}{
\begin{smallmatrix}2 \text{ \it entries}\end{smallmatrix}
}
\\
 a_1&a_1&\ldots & a_{q}&a_{q} &-1&\ldots&-1&0&-1\\
 -1&-1&\ldots & -1&-1 &0&\ldots&0&-1&0\\
\end{array}
\end{equation}
Let $(z_1,\ldots, z_{2r+2q+1})$ be the co-ordinates on $\C^{2r+2q+1}$.
The action of $(\C^*)^2$ on $\C^{2r+2q+1}$ commutes with the action of 
$$G=SU(2)^{q}\times SU(2r),$$ where the $SU(2)$ factors act respectively on $(z_1,z_2),\ldots,(z_{2q-1},z_{2q})$, and $SU(2r)$ acts on $(z_{2q+1},\ldots,$ $z_{2q+2r-1},z_{2q+2r+1})$, note we have omitted $z_{2q+2r}$. (If we view $X$ as a projective bundle over $\CP^{2r-1}$ as in (\ref{eq:define_X}), the co-ordinates on which $SU(2r)$ acts are the homogeneous co-ordinates on the base.) 
Denote by $G^\R=SO(2)^q\times SO(2r)$ the real form of $G$.
Because all $a_i$ are odd, the action of $(-1,+1)\in (\C^*)^2$ coincides with the action of $-I\in G$.
Consequently, the action of $G$
descends to a Hamiltonian  action of $G/{\pm I}$ on $X$. Its real form $G^\R/\pm I$ preserves the real Lagrangian $L\subset X$, and we let $\gamma$  be the $S^1$-subgroup of $G^\R/\pm I$ defined as follows. This subgroup lifts to the  path from $I$ to $-I$ in $G^\R$ which is the image of the rotation 
$\left( \begin{smallmatrix}
  \cos t&\sin t\\
-\sin t&\cos t
 \end{smallmatrix}
\right)\in SO(2)$, $t\in [0,\pi]$,
under the diagonal inclusions $$SO(2)\subset SO(2)^q\times SO(2)^r\subset SO(2)^q\times SO(2r)=G^\R.$$ 
Recall that we are assuming $\chr\k=2$.

\begin{lemma}
\label{lem:hirzebruch_orbit_homology}
 The homology class of $\gamma$-orbits on $L$ is non-zero in $H_1(L;\k)$.
\end{lemma}
\begin{proof}
Indeed, $L$ is a real projective bundle over $\RP^{2r-1}$, and the orbits project to the non-trivial cycle on the base, provided $\chr\k=2$. 
\end{proof}

\begin{lemma}
\label{lem:hirzebruch_compute_seidel_gamma}
 We have $\S(\gamma)+1=y^r(x+y)^{-p}+1$ (which is the generator of $\ker\F$ from Lemma~\ref{lem:hirzebruch_ker_f_generator}).
\end{lemma}

\begin{proof}
Inside the complex group $G/\pm I$, the loop $\gamma$ is homotopic to the loop
$\gamma'$ lifting to the path from $I$ to $-I$ in $G$ which is the image of the path
$\left(\begin{smallmatrix}
  e^{it}&0\\
0&e^{-it}
 \end{smallmatrix}
\right)\in SU(2)$, $t\in [0,\pi]$, under the diagonal inclusions $$SU(2)\subset SU(2)^q\times SU(2)^r\subset SU(2)^q\times SU(2r)=G.$$  By using the action of $\C^*\subset (\C^*)^2$ corresponding to the first vector in (\ref{eq:torus_weights}), we see that $\gamma'$ descends to the same Hamiltonian loop in $X$ as the loop $\gamma''$ in $G$ which  acts on $\C^{2r+2q+1}$ as follows: 
\begin{multline*}
 (z_1,\ldots,z_{2r+2q+1})\mapsto (e^{it\frac{a_1+1}2}z_1,e^{it\frac{a_1-1}2}z_2,\ldots,e^{it\frac{a_q+1}2}z_{2q-1},e^{it\frac{a_q-1}2}z_{2q},\\
z_{2q+1},e^{-it}z_{2q+2}\ldots,e^{-it}z_{2r+2q-2},z_{2r+2q-1},z_{2r+2q},e^{-it}z_{2r+2q+1}), \quad t\in[0,2\pi].
\end{multline*}
Note that here $t$ runs through $[0,2\pi]$, hence the $\frac 1 2$-factors. Because all $a_i$ are odd, $\gamma''$ is now a closed loop in $G$, not only in $G/\pm I$. So by \cite{MDT06} its Seidel element  $\S(\gamma'')\in QH^*(X)$ can be computed as the quantum product of powers of the divisors  corresponding to the co-ordinates on $\C^{2r+2q+1}$, where the powers are the multiplicities of  rotations. Given $\chr\k=2$, and recalling that $\S(\gamma'')=\S(\gamma')=\S(\gamma)$, we get:
$$
\S(\gamma)=(x+y)^{\frac{a_1+1}2}(x+y)^{\frac{a_1-1}2}\ldots (x+y)^{\frac{a_q+1}2}(x+y)^{\frac{a_q-1}2}y^{-1}\ldots y^{-1}=(x+y)^py^{-r}.
$$
This element squares to 1 by (\ref{eq:QH_X_p_q}) (in agreement with the fact $\gamma$ has order 2 in $\pi_1(G/\pm I)\cong\Z/2$), so it also equals $y^{r}(x+y)^{-p}$, which proves Lemma~\ref{lem:hirzebruch_compute_seidel_gamma}.
\end{proof}

We conclude the proof of Theorem~\ref{th:hirzebruch}(b). By Lemmas~\ref{lem:hirzebruch_ker_f_generator} and~\ref{lem:hirzebruch_compute_seidel_gamma}, $\ker\F$ is the ideal generated by $\S(\gamma)+1$. Suppose $P\in QH^*(X)$ such that $\CO^*(P)=0\in HH^*(L,L)$. Then $\CO^0(P)=0$, so $P\in \ker\F$ by Theorem~\ref{th:toric_CO_0}. Consequently $P=(\S(\gamma)+1)*Q$, and if $P\neq 0$ then $Q\notin\ker\F$ (because otherwise we would get $P\in (\ker\F)^2=\{0\}$). Apply Theorem~\ref{th:CO_invt_lag}(b) to the product $(\S(\gamma)+1)*Q$; the left hand side of $(**)$ vanishes because $\mu^2$ is commutative on $HF^*(L,L)$ \cite{Ha13}, and the right hand side is non-trivial for some $y$ by Lemma~\ref{lem:hirzebruch_orbit_homology} and because $\CO^0(Q)\neq 0$. It follows that $\CO^*(P)\neq 0$. We have shown that $\CO^*$ is injective,  and split-generation  follows from Theorem~\ref{th:generation_crit}.
 Note that $w(L)=0$ holds for all real Lagrangians, as Maslov 2 disks come in pairs because of the action of the anti-holomorphic involution, see \cite{Ha13}.
\end{proof}

The following corollary in particular implies Proposition~\ref{prop:CO_BlP9_Inject} from the introduction.

\begin{corollary}
\label{cor:split_gen_Bl_2q-1}
 Let $X=Bl_{\CP^{2q-1}}\CP^{2r+2q-1}$, and $L\subset X$ be the real Lagrangian (diffeomorphic to $Bl_{\RP^{2q-1}}\RP^{2r+2q-1}$). Assume $\gcd(2q+1,2r-2q)\ge 2$ and that either $r$ or $q$ are odd.
Then $\CO^*\co QH^*(X)\to HF^*(L,L)$ is injective, although $\CO^0$ is not. Consequently, $L$ split-generates $\Fuk(X)_0$.
\end{corollary}

\begin{proof}[Proof of Proposition~\ref{prop:CO_BlP9_Inject}] 
Take $X$ as in (\ref{eq:define_X}) with
 $a_1=\ldots=a_k=1$, then $X=Bl_{\CP^{k-1}}\CP^n$, see e.g.~\cite[Proposition 11.14]{EHBook}. 
 The additional hypotheses of the current corollary make sure $X$ satisfies all conditions of Theorem~\ref{th:hirzebruch}(b), which together with the split-generation criterion (Theorem~\ref{th:generation_crit}(b)) implies the corollary.
\end{proof}

In order to deduce non-displaceability results between the real Lagrangian $L$ and other Lagrangians with arbitrary obstruction numbers, we need the following lemma.

\begin{lemma}
\label{lem:split_gen_prod}
 Suppose $\chr \k=2$, $L$ is an object of $\Fuk(X)_w$ and $\CO^*\co QH^*(X)\to HH^*(L,L)$ is injective. Assuming Hypothesis~\ref{hyp:ganatra} below, $L\times L$ split-generates $\Fuk(X\times X)_0$.
\end{lemma}

Note that by Lemma~\ref{lem:eigenvalue_split}, the condition of Lemma~\ref{lem:split_gen_prod} can only hold if $QH^*(X)=QH^*(X)_w$ or  $L$ is non-orientable.

\begin{proof}
First, observe that $w(L\times L)=2w(L)=0$.
 By \cite{She13}, the injectivity of $\CO^*$ is equivalent to the fact that the open-closed map $\OC^*\co HH_*(L,L)\to QH^*(X)$ hits the unit $1\in QH^*(X)$. 
 
 \begin{hypothesis}		
 	\label{hyp:ganatra}
 		There is a commutative diagram
 		$$
 		\xymatrix{
 			HH_*(L,L)\otimes HH_*(L,L)\ar[r]\ar[d]^{\OC^*\otimes \OC^*}& HH_*^{split}(L\times L,L\times L)\ar[d]^{\OC^*_\text{prod}}
 			\\
 			HF^*(X)\otimes HF^*(X)\ar[r]_{=}& HF^*(X\times X)
 		}
 		$$
 		where $\OC^*_\text{prod}$ is the open-closed map on the product, and the $HF^*$ are Hamiltonian Floer cohomologies, isomorphic to the quantum cohomologies of the corresponding spaces.
 \end{hypothesis}
 Here $HH_*^{split}(L\times L,L\times L)$ indicates that the \ai structure on $L\times L$ is computed using a split Hamiltonian perturbation and a product almost complex structure; such a choice can be made regular. 
 We expect the hypothesis to hold 
 following Ganatra \cite[Remark~11.1]{Ga13} who stated it on chain level, in the setup of the wrapped Fukaya category of an exact manifold. A slight complication is that the \ai algebra of $L\times L$ appearing in the top right corner of the diagram had to be equipped with so-called one-sided homotopy units; their presense is denoted by a tilde in \cite[Remark~11.1]{Ga13}. This does not affect the diagram on the homology level \cite[Proposition~10.10]{Ga13}, but we have not checked how this subtlety carries over to the monotone setup; therefore we leave Hypothesis~\ref{hyp:ganatra} as a conjecture. 
 
 Given Hypothesis~\ref{hyp:ganatra}, if $\OC^*$ hits the unit, then $\OC^*\otimes \OC^*$ and $\OC^*_\text{prod}$ also do. The latter fact implies that $\CO^*$ is injective on the product, and  split-generation follows from Theorem~\ref{th:generation_crit}(b).
\end{proof}

Corollary~\ref{cor:non_disp_Bl_P9} from the introduction is a particular case of the following.

\begin{corollary}
 Let $\k$ be a field of characteristic 2 and $L\subset X$ be as in Theorem~\ref{th:hirzebruch}(a) or~(b), or as in Corollary~\ref{cor:split_gen_Bl_2q-1}. Suppose $L'\subset X$ another monotone Lagrangian, perhaps equipped with a local system $\pi_1(L)\to \k^\x$, with minimal Maslov number at least 2 and such that $HF^*(L',L')\neq 0$. If $w(L')\neq 0$, assume Hypothesis~\ref{hyp:ganatra}. Then $L\cap L'\neq\emptyset$. 
\end{corollary}

\begin{proof}
 If $w(L')=0$, this follows from the fact $L$ split-generates $\Fuk(X)_0$ and Lemma~\ref{lem:split_gen_implies_non_disp}. If $w(L')\neq 0$, we  have that $w(L'\times L')=2w(L')=0$, so $L'\times L'$ is an object of $\Fuk(X\times X)_0$ which is split-generated by $L\times L$ by Lemma~\ref{lem:split_gen_prod}. Then $(L\times L)\cap (L'\times L')\neq 0$ by Lemma~\ref{lem:split_gen_implies_non_disp}, and so $L\cap L'\neq 0$.
\end{proof}

\subsection{An application to non-formality}
Recall that if $\A\to \A'$ is a quasi-iso\-morph\-ism of \ai categories, it induces an isomorphism $HH^*(\A)\to HH^*(\A')$, see e.g.~Seidel \cite[(1.14)]{SeiFlux}. We will need an explicit chain-level formula for this isomorphism, which can be obtained by combining Seidel's argument with Ganatra's functoriality formulas \cite[Section~2.9]{Ga13}, and this requires a short account. We are  assuming the reader is familiar with the basic language of \ai categories from e.g.~\cite{SeiBook08, She13, Ga13}, so that we can skip some basic definitions and present the other ones rather informally. For simplicity, we are working with $\chr\k=2$ so we won't have to worry about signs, and restrict to \ai algebras rather than categories.

Recall that if $\A$ is an \ai algebra, its Hochschild cohomology $HH^*(\A)$ can be seen as Hochschild cohomology $HH^*(\A,\A)$ of $\A$ as an $\A-\A$ bimodule. If $F\co \A\to\A'$ is a quasi-isomorphism between \ai algebras, it induces quasi-isomorphisms
\begin{equation}
\label{eq:two_hh_isoms}
CC^*(\A,\A)\xrightarrow{F_*}CC^*(\A,F^*\A')\xleftarrow{F^*}CC^*(\A',\A'), 
\end{equation}
 which proves that $HH^*(\A,\A)\cong HH^*(\A',\A')$.
Chain-level formulas for the two intermediate quasi-isomorphisms, which we will now recall, were written down e.g.~by  Ganatra \cite[Section~2.9]{Ga13} (in the context of Hochschild homology, but these are easily adjusted to cohomology).

If $\B,\B'$ are two  $\A-\A$ bimodules, a morphism $G\co \B\to\B'$ is a sequence of maps $G^k\co \A^{\ot i}\ot \B\ot \A^{\ot j}\to \B'$, $i+j+1=k$, satisfying a sequence of relations which we informally write down as $$\sum_{\st} G^\st(\id^{\ot \st}\ot \mu_{\A \text{ \it or } \B}^\st\ot \id^{\ot \st})=\sum_\st\mu_{\B'}^\st(\id^{\ot \st}\ot G^\st\ot \id^{\ot \st}).$$
 Here $\st$ are positive integers which are mutually independent but are such that the total number of inputs on both sides of the equation is the same; the sum is over all such possibilities; and the structure map on the left is $\mu_\A^\st$ or $\mu_\B^\st$ depending on whether one of its arguments is in $\B$. 
 In its full form, the above relation should be written as follows:
 $$
 \footnotesize
 \begin{array}{l}  
 \displaystyle \sum_{i_1+i_2+i_3+i_4=l}G^{i_1+i_4-i_3+1}(a_1,\ldots,a_{i_1},\mu_\B^{i_3-i_1+1}(a_{i_1+1},\ldots,a_{i_2},b,
 a_{i_2+1},\ldots a_{i_3}),
 a_{i_3+1},
  \ldots,a_{i_4})
 \\
 + 
 \displaystyle \sum_{i_1+i_2+i_3+i_4=l}G^{i_1+i_4-i_2+2}(a_1,\ldots,a_{i_1},\mu_\A^{i_2-i_1}(a_{i_1+1},\ldots,a_{i_2}),a_{i_2+1},\ldots,a_{i_3},b,
 a_{i_3+1},\ldots a_{i_4})
 \\
   + 
   \displaystyle \sum_{i_1+i_2+i_3+i_4=l}G^{i_2+i_4-i_3+2}(a_1,\ldots,a_{i_1},b,a_{i_1+1},\ldots,a_{i_2},\mu_\A^{i_3-i_2}(a_{i_2+1},\ldots,a_{i_3}),
   a_{i_3+1},\ldots a_{i_4})
   \\
     =\displaystyle \sum_{i_1+i_2+i_3+i_4=l}\mu_{\B'}^{i_1+i_4-i_3+1}(a_1,\ldots,a_{i_1},G^{i_3-i_1+1}(a_{i_1+1},\ldots,a_{i_2},b,a_{i_2+1},\ldots,a_{i_3}),a_{i_3+1},\ldots,a_{i_4}),
 \end{array}
 $$
 for $a_i\in\A$ and $b\in \B$.
 We will keep the informal style of notation, in which the inputs are omitted and the valencies are replaced by $\st$, further. The induced map $G_*\co CC^*(\A,\B)\to CC^*(\A,\B')$ is defined by
\begin{equation}
 \label{eq:hh_module_morphism_pushforward}
(G_*(h))^\st=\sum\nolimits_\st G^\st(\id^{\ot\st}\ot h^\st\ot\id^{\ot\st})
\end{equation}
where $h^\st\co \A^{\ot \st}\to \B$ and $(G_*(h))^\st\co \A^{\ot \st}\to \B'$. If $G$ is a quasi-isomorphism, so is $G_*$.

If $\A,\A'$ are two \ai algebras, a morphism $F\co \A\to\A'$ is a sequence of maps $F^\st\co \A^{\ot\st}\to\A'$ such that 
$$\sum\nolimits_\st \mu_{\A'}^\st(F^\st\ot\ldots\ot F^\st)=\sum\nolimits_\st F^\st(\id^{\ot\st}\ot \mu_\A^\st\ot\id^{\ot\st}).$$
Next, if $\B$ is an $\A'-\A'$ bimodule, its two-sided pull-back $F^*\B$ is an $\A-\A$ bimodule based on the same vector space $\B$, whose structure maps are \cite[Section~2.8]{Ga13}
\begin{equation}
\label{eq:hh_module_pullback}
 \mu^\st_{F^*\B}=\sum\nolimits_\st \mu_\B^\st(F^\st\ot \ldots\ot F^\st\ot\id_\B\ot F^\st\ot\ldots\ot F^\st)
\end{equation}
There is also a morphism $F^*\co CC^*(\A',\B)\to CC^*(\A,F^*\B)$ defined by
\begin{equation}
\label{eq:hh_algebra_morphism_pullback}
 (F^*(h))^\st=\sum\nolimits_\st h^\st(F^\st\ot \ldots\ot F^\st)
\end{equation}
where $h^\st\co (\A')^{\ot\st}\to\B$ and $(F^*(h))^\st\co \A^{\ot\st}\to\B$. The total number of inputs here can be zero, and $F^*(h)^0=h^0$.
If $F$ is a quasi-isomorphism, so is $F^*$.

If, again, $F\co \A\to \A'$ is a morphism of \ai algebras, let $F^*\A'$ be the $\A-\A$ bimodule which is the pull-back of $\A'$ seen as an $\A'-\A'$ bimodule. 
\begin{lemma}
\label{lem:morphism_alg_and_mod}
The same sequence of maps $F^\st\co \A^{\otimes \st}\to\A'$ provides a morphism of $\A-\A$ bimodules $\A\to F^*\A'$, also denoted by $F$. 
\end{lemma}
\begin{proof}
 We must check that
$\sum_\st F^\st(\id^{\ot\st}\ot\mu_\A^\st\ot\id^{\ot\st})=
\sum_\st \mu_{F^*\A}^\st(\id^{\ot\st}\ot F^\st\ot\id^{\ot\st})$. If we apply formula (\ref{eq:hh_module_pullback}) to rewrite the right hand sum, the unique $\id$-factor in (\ref{eq:hh_module_pullback}), which in our case is $\id_{\A'}$, gets applied to the $F^\st$-factor. So our right hand sum equals
$\sum_\st \mu_{\A'}^\st(F^\st\ot\ldots \ot F^\st\ot \ldots \ot F^\st)$ which is exactly the condition that $F$ is a morphism of \ai algebras $\A\to\A'$.
\end{proof}
This lemma explains the precise meaning of (\ref{eq:two_hh_isoms}): if $F\co \A\to \A'$ is a morphism of \ai algebras, then the first map $F_*$ from (\ref{eq:two_hh_isoms}) is the push-forward of $F$ considered as a morphism of modules $\A\to F^*\A'$, given by formula (\ref{eq:hh_module_morphism_pushforward}). The second map in (\ref{eq:two_hh_isoms}) is the pull-back as in (\ref{eq:hh_algebra_morphism_pullback}).

Next, recall that a general property of quasi-isomorphisms between \ai algebras (bimodues, etc.) is that they have quasi-inverses \cite[Chapter~1]{SeiBook08}.
If $\F\co \A\to\A'$ is a quasi-isomorphism, then the $\A-\A$ bimodule morphism $F$ from Lemma~\ref{lem:morphism_alg_and_mod} is also a quasi-isomorphism, hence there is an $\A-\A$ bimodule quasi-isomorphism $G\co F^*\A'\to\A$ which is a quasi-inverse of $F$, so we have quasi-isomorphisms:
\begin{equation}
\label{eq:hh_G_and_F}
CC^*(\A,\A)\xleftarrow{G_*}CC^*(\A,F^*\A')\xleftarrow{F^*} CC^*(\A',\A').
\end{equation}
Their composition acts on Hochschild cochains by:
$$
(G_*F^*(h))^\st=\sum\nolimits_\st G^\st(\id^{\ot \st}\ot h^\st(F^\st\ot\ldots\ot F^\st)\ot \id^{\ot\st})
$$
where $h^\st\co (\A')^{\ot \st}\to\A'$ and $(G_*F^*(h))^\st\co \A^{\ot\st}\to\A$.
In particular, $(G_*F^*(h))^0=G^1(h^0)$, and 
if  $h^0=0\in\A'$ then
\begin{equation}
\label{eq:hh_comp_G_F_1}
(G_*F^*(h))^1(u)=G^1(h^1(F^1(u))),\quad u\in \A'.
\end{equation}
Note that $G^1\co \A'\to\A$, $F^1\co \A\to\A'$ are chain maps with respect to $\mu^1_\A$, $\mu^1_{\A'}$ and are cohomology inverses of each other.

Assume $L\subset X$ is a Lagrangian  preserved by a Hamiltonian loop $\gamma$ which together satisfy the conditions of Theorem~\ref{th:CO_invt_lag} (those conditions which are common to all parts of the theorem).
Assume the \ai algebra $CF^*(L,L)$ is formal, i.e.~there is an \ai quasi-isomorphism $F\co HF^*(L,L)\to CF^*(L,L)$. Denote 
\begin{equation}
\label{eq:h_formal}
h\coloneqq (G_*F^*)(\CO^*(\S(\gamma))-1) \in CC^*(HF^*(L,L),HF^*(L,L)),
\end{equation}
where $G$ is a quasi-inverse of $F$, and $G_*,F^*$ are as in (\ref{eq:hh_G_and_F}). So $h$ is a Hochschild cochain for the accociative algebra $HF^*(L,L)$. Then by Corollary~\ref{cor:compute_CO_0}, $h^0=0$,  Proposition~\ref{prop:compute_CO_1} and equation (\ref{eq:hh_comp_G_F_1}),
$$h^1(x)=\rho(l)\cdot \langle \Psi(F^1(x)),l\rangle\cdot G^1(1_L).$$ 
(We have dropped the extra sign, working in characteristic~two.)
Let us additionally assume that $L$ is \emph{wide} \cite[Definition~1.2.1]{BC09A}, i.e.~there is a vector space isomorphism between $H^*(L)$ and $HF^*(L,L)$, and that $L$ admits a perfect Morse function. These conditions enable us to identify $CF^*(L,L)\cong HF^*(L,L)$ as vector spaces.
Because $G^1$ is cohomologically unital, $G^1(1_L)=1_L\in HF^*(L,L)$, so
\begin{equation}
\label{eq:h_1_formal}
h^1(x)=\rho(l)\cdot \langle \Psi(F^1(x)),l\rangle\cdot 1_L\in HF^*(L,L).
\end{equation}
Under our identifications,  $\Psi$ becomes an isomorphism between the vector spaces below, and $F^1$ can be considered as algebra isomorphism from $HF^*(L,L)$ to itself:
\begin{equation}
\label{eq:F_1_and_Psi}
 HF^*(L,L)\xrightarrow[\text{\it algebra iso.}]{F^1} HF^*(L,L)\xrightarrow[\text{\it v.~space iso.}]{\Psi} H^*(L).
\end{equation}

We now turn the discussion to Hochschild cohomology of monic algebras. Let $f(u)\in \k[u]$ be a polynomial and $A\coloneqq\k[u]/(f)$ be the quotient algebra; it is called a monic algebra. This is an algebra in the ordinary associative sense; we can  consider it as an \ai algebra by equipping it with trivial higher structure maps. The Hochschild cohomology algebra $HH^*(A)$ was computed by Holm~\cite{Holm01}. Recall that Hochschild cohomology of ungraded associative algebras is $\Z$-graded (unlike Hochschild cohomology of non-$\Z$-graded \ai algebras): cochains $A^{\otimes k}\to A$ are said to have degree $k$, and the differential has degree 1.
By \cite[Proposition~2.2]{Holm01}, 
$$HH^k(A)=
\begin{cases}
 A,& \mbox{if } k=0\\
Ann_A(f')&\mbox{if } k>0 \mbox{ is odd}\\
A/(f')&\mbox{if } k>0 \mbox{ is even}.
\end{cases}
$$
Although we already know the answer, let us compute $HH^1(A)$ explicitly, as this will be helpful later.

\begin{lemma}
For an $h\co A\to A$ which is a Hochschild cocycle in $CC^1(A)$,  we must necessarily have
\begin{equation}
\label{eq:HH_1_comm}
h(u^m)=amu^{m-1}
\end{equation}
for some fixed $a\in A$. Note that $a=h(u)$.
\end{lemma}
\begin{proof}
Let us compute the Hochschild differential $\bd h\co A\otimes A\to A$ on the two elements: $u$ and $u^{m-1}$, for some $m\in\N$. Because $\bd h$ vanishes by assumption, we get:
$$
0=(\bd h)(u, u^{m-1})=uh(u^{m-1})+u^{m-1}h(u)+h(u^m).
$$
The desired formula follows by induction on $m$.
\end{proof}	

So any cocycle $h\in CC^1(A)$ is completely determined by a single element $a=h(u)\in A$, and $h$ must meet an additional condition that $h(f(u))=h(0)=0$, which is equivalent to $a\in Ann_A(f')$. As the differential $CC^0(A)\to CC^1(A)$ vanishes, we get an isomorphism $HH^1(A)\to Ann_A(f')$, $h\mapsto h(u)$.

We will further assume that $\chr\k=2$ and $f'=0$. The latter condition means that $f$ is a sum of even powers of $u$. Denote by 
$$
\psi\co HH^1(A)\to A
$$
the isomorphism $\phi(h)=h(u)$ from above. Note that if $s(u)\in A$ is an arbitrary element given by a polynomial with derivative $s'(u)$, then by (\ref{eq:HH_1_comm}) we get
\begin{equation}
\label{eq:hh_psi_formula} 
\psi(h)=s'(u)\cdot h(s(u)).
\end{equation}
For $k> 1$, we also have isomorphisms $\psi\co HH^k(A)\to A$, all of which we denote by the same letter by abusing notation; we will not need an explicit formula for these isomorphisms when $k>1$.

Moreover,~\cite[Lemma~5.1]{Holm01} computes the Yoneda product on $HH^*(A)$. In particular, given  $h_1,h_2\in HH^1(A)$, their Yoneda product $h_1\st h_2$ is determined by
\begin{equation}
\label{eq:product_holm}
\psi(h_1\st h_2)=\psi(h_1)\cdot \psi(h_2)\sum\nolimits_{j\, odd} f_{2j}u^{2j-2}\in A,
\end{equation}
where $f=\sum_j f_ju^j$, $f_j\in \k$.

The two strands of discussion can be combined in the following theorem.

\begin{theorem}
\label{th:invt_non_formal}
 Let $\k$ be a field of characteristic 2, $L\subset X$  a Lagrangian preserved by a Hamiltonian loop $\gamma$ which together satisfy the conditions of Theorem~\ref{th:CO_invt_lag} (those conditions which are common to all parts of the theorem). 
Assume there is an algebra isomorphism $HF^*(L,L)\cong \k[u]/(f)$ where $f(u)=\sum_{j\ge 0}f_ju^j$ is a polynomial, and also that $L$ is wide and admits a perfect Morse function, so that we can identify the vector spaces $HF^*(L,L)\cong CF^*(L,L)$, and $\Psi\co HF^*(L,L)\to H^*(L)$ becomes an isomorphism of vector spaces. Further, assume:

\begin{itemize}
 \item $f'=0$, and  $\sum_{j\, odd}f_{2j}u^{2j-2}$ is invertible in $\k[u]/(f)$;
\item $\langle \Psi(r(u)),l\rangle =1$ for an element $r(u)\in \k[u]/(f)\cong HF^*(L,L)$ which generates $HF^*(L,L)$ as an algebra;
\item $\S(\gamma)^2=1\in QH^*(X)$.
\end{itemize}
Then the Fukaya \ai algebra of $L$ is not formal over $\k$.
\end{theorem}

\begin{proof} 
Supposing $CF^*(L,L)$ is formal, let $h$ be as in (\ref{eq:h_formal}) and $F^1$ be as in (\ref{eq:F_1_and_Psi}). Then there exists $s(u)\in HF^*(L,L)$ (we  view this element as a polynomial in $\k[u]/(f)$) such that $F^1(s(u))=r(u)$. Then by (\ref{eq:h_1_formal}) $h^1(s(u))=\rho(l)\cdot 1\in HF^*(L,L)$, so by (\ref{eq:hh_psi_formula}), $$\psi(h^1)=\rho(l)\cdot s'(u)\in HF^*(L,L).$$ 
Further, note that $h\star h=0$ because $(\S(\gamma)+1)^2=0$, so (\ref{eq:product_holm}) yields
$$\rho(l)^2\cdot (s'(u))^2\sum_{j\, odd}f_{2j}u^{2j-2}=0\in HF^*(L,L).$$  By hypothesis, this implies $(s'(u))^2=0$,
so $s'(u)\in \ker \F$ where $\F\co \k[u]/(f)\to \k[u]/(f)$ is the Frobenius endomorphism. In general, over $\chr\k=2$
it is always true that $s'(u)$ is a sum of even powers of $u$, so $s'(u)$
is a square of another polynomial: $s'(u)=(t(u))^2$. Then $t(u)^2\in\ker\F$, which implies $t(u)\in\ker\F$ because $\ker\F$, being an ideal in $\k[u]/(f)$, is necessarily prime. Consequently, $s'(u)=0$.
So $s(u)$ is a sum of even powers of $u$,
hence the subalgebra generated by $s(u)$ lies in the subalgebra of $\k[u]/(f)$ generated by $u^2$, which
is smaller than the whole $\k[u]/(f)$: for example, it does not contain the element $u$. (Recall that $f$ is also a sum of even powers of $u$.) On the other hand, we know that $F^1$ is an algebra isomorphism, $F^1(s(u))=r(u)$ and $r(u)$ generates the whole $HF^*(L,L)$ by hypothesis, so $s(u)$ should also generate $HF^*(L,L)$, which is a contradiction. 
\end{proof}

\begin{proof}[Proof of Proposition~\ref{prop:RPn_not_formal}]
Take the real loop $\gamma$ preserving $\RP^{4n+1}$ defined in the proof of Proposition~\ref{prop:CO_RPn_Inject} (see Section~\ref{sec:intro}), and denote  $L=\RP^{4n+1}$, $X=\CP^{4n+1}$.
Recall that, if $x\in H^2(X)$ is the generator, then $QH^*(X)\cong \k[x]/(x^{4n+2}+1)$ and $\S(\gamma)=x^{2n+1}$, so $\S(\gamma)^2=1$. Also recall that $l\in H_1(L)\cong \k$ is non-zero.
By Theorem~\ref{th:haug}, $HF^*(L,L)\cong \k[u]/(u^{4n+2}+1)$ where $u\in CF^1(L,L)\cong \k$, and we have $\langle \Psi(u),l\rangle=1$.
Now apply Theorem~\ref{th:invt_non_formal} taking $r(u)=u$ to conclude the proof.
\end{proof}

\begin{proposition}
 Let $X=Bl_{\CP^{2q-1}}\CP^{2r+2q-1}$, $L\subset X$ be the real Lagrangian (diffeomorphic to $Bl_{\RP^{2q-1}}\RP^{2r+2q-1}$). Assume that $\gcd(2q+1,2r-2q)\ge 2$ and that either $r$ or $q$ are odd.
Then the \ai algebra of $L$ is not formal  over a characteristic 2 field.
\end{proposition}

\begin{proof} 
We recall that all real Lagrangians are wide by Theorem~\ref{th:haug} and admit a perfect Morse function by~\cite{Ha13}.
 The fact that $\gcd(2q+1,2r-2q)\ge 2$ means we are in the situation of Theorem~\ref{th:hirzebruch}(b), with $k=2q=2p$, $a_1=\ldots=a_k=1$. We have already seen (\ref{eq:QH_Hirzebruch_poly}) that $HF^*(L,L)\cong \k[u]/(f)$ with $f'=0$, and it is easy to check that $\sum_{j\ odd}f_{2j}u^{2j-2}$ is invertible provided that either $r$ or $q$ is odd (otherwise this element would vanish). Moreover, via (\ref{eq:phi_acting_on_u}) and Haug's isomorphism (Theorem~\ref{th:haug}), $u^{p/g}$ corresponds to the generator of $CF^1(L,L)\cong \k^2$ such that $\langle\Psi(u),l\rangle=1$. Now apply Theorem~\ref{th:invt_non_formal} taking $r(u)=u^{p/g}$.
\end{proof}

\subsection{Non-formality of the equator on the sphere}
Proposition~\ref{prop:RPn_not_formal}  says in particular that the \ai algebra of an equatorial circle on $S^2$ is not formal over $\chr \k=2$. This is an especially simple case which can be verified by hand, and it is worth discussing it in more detail.
Let $L_1\subset S^2$ be a fixed equator, and $L_2,L_3,\ldots$ be a sequence of its small Hamiltonian perturbations; assume $|L_i\cap L_j|=2$ for each
$i,j$. Then 
$CF^0(L_i,L_j)\cong \k$ is generated by an element which we denote by $1$ (this is the cohomological unit), and $CF^1(L_i,L_j)\cong \k$ is generated by an element which we denote by $u$ (we use the same letter for all $i,j$). Of the two intersection points $L_i\cap L_j$, the point $u$ is the one at which $T_uL_j$ is obtained from $T_uL_i$ by a small positive rotation with respect to the $\omega$-induced orientation on $S^2$. 
Consider the \ai structure maps between the consequtive Lagrangians:
\begin{equation}
\label{eq:mu_k_order}
\mu^k\co  CF^*(L_{k},L_{k+1})\otimes\ldots\otimes CF^*(L_1,L_2)\to CF^*(L_1,L_{k+1}) 
\end{equation}
given by counting immersed polygons as in~\cite{SeiBook08}.
These give a model for the \ai algebra of $L$, because all the $L_i$ differ small perturbations and we can canonically identify the spaces $CF^*(L_i,L_{i+1})$ with each other. The  \ai maps will depend on the particular arrangement of the $L_i$, although up to quasi-isomorphism they give the same \ai algebra.

\begin{remark}
	The fact the \ai algebra of $L$ defined using the count of polygons is quasi-isomorphic to the one defined using Hamiltonian perturbations seems not to have been written down in detail but is widely accepted. An approach is sketched in \cite[Remark~7.2]{Sei11}, and also performed in \cite{She11} in a slightly different setup.
\end{remark}

Let us compute some of the \ai structure maps using a specific choice of the $L_i$. Fix a Hamiltonian $H$ whose flow is the rotation of $S^2\subset \R^3$ around an axis which is not orthogonal to the plane intersecting $S^2$ along the equator $L_1$. Let $L_2,L_3,\ldots$ be obtained from $L_1$ by applying that rotation by small but consequtively increasing angles, i.e.~$L_i$ are time-$t_i$ push-offs of $L_1$ under the flow of $H$, $0=t_1<t_2<t_3\ldots$. The first four resulting circles $L_i$ are represented in Figure~\ref{fig:Circle_A_infty}(a). The pairwise intersections of the $L_i$ are contained in two opposite patches of the sphere; those patches are shown in the top and bottom of Figure~\ref{fig:Circle_A_infty}(a) together with the $L_i$ on them, which are depicted by straight lines. Both patches are drawn as if we look at them from the same point ``above'' the sphere, so that the positive rotation (with respect to the orientation on $S^2$) is counter-clockwise on the upper patch and clockwise on the lower patch. 
For this particular choice of  perturbations, and for each $i<j$, all degree-one points $u\in CF^*(L_i,L_j)$ are located on the upper patch, and all points $1\in CF^*(L_i,L_j)$ are on the lower patch.

\begin{figure}[h]
 \includegraphics{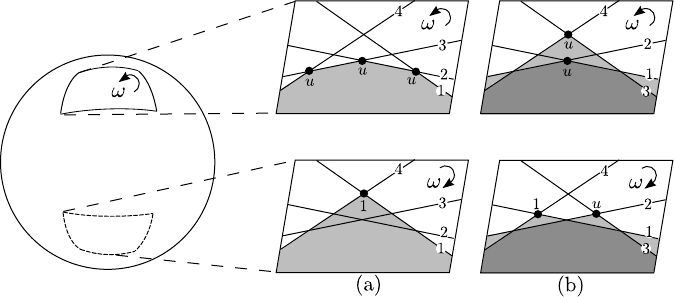}
\caption{Two different configurations (a) and (b) consisting of four small Hamiltonian push-offs $L_1,\ldots,L_4$ (marked by numbers) of an equatorial circle on $S^2$.
The image of the disk contributing to $\mu^3(u,u,u)=1$ is shaded.
}
\label{fig:Circle_A_infty}
\end{figure}

We claim that in this model we get:
$$
\label{eq:mu_k_S1}
\mu^3(u,u,u)=1,\ \mu^3(u,u,1)=0,\ \mu^3(u,1,u)=0, \ \mu^3(1,u,u)=0.
$$
 For grading reasons, $\mu^k(u,\ldots,u)$ is a multiple of $1$, and is determined by counting Maslov index 2 disks. There is a unique such disk; for $k=3$ it is shown in gray shade in Figure~\ref{fig:Circle_A_infty}(a) on the two patches; away from the patches this disc is just a strip between $L_1$ and $L_4$. 
Also for grading reasons, the only other products which can possible be non-trivial are $\mu^k(u,\ldots,u,1,u\ldots,u)\in\{0,u\}$, where exactly one input is $1$. It possible to check that these vanish for our configuration of the circles $L_i$, at least when $k=3$.
Now note that 
$$
\mu^2(1+u,1+u)=0,\quad \mu^3(1+u,1+u,1+u)=1.
$$
The latter equality exhibits a non-trivial Massey product, seen as a well-defined element of $$\k[u]/(1+u)\cong\k.$$ 
An explanation where the Massey products generally belong to is found in \cite[Remark~1.2]{SeiBook08} and explains the quotient by $1+u$ above. 
The presence of a non-trivial Massey product is invariant under quasi-isomorphisms. To see this, recall that the analogous fact for dg algebras is easy, and any \ai algebra is quasi-isomorphic to a dg algebra. Moreover, the Massey products for the \ai and dg models satisfy a simple relation \cite[Theorem~3.1 and Corollary~A.5]{LPWZ09}, in particular, if triple Massey products of an \ai algebra are non-trivial, they remain non-trivial for its dg-model.
This gives us an alternative proof of the fact that
the \ai algebra of the equator on $S^2$ is not formal.

For any other arrangement of the $L_i$,
we will necessarily have $\mu^3(1+u,1+u,1+u)=1$ modulo $1+u$   
because of invariance of Massey products,
meaning that $$\mu^3(1+u,1+u,1+u)\in\{1,u\}.$$
For example, another possible configuration of $L_1,\ldots,L_4$ is shown in Figure~\ref{fig:Circle_A_infty}(b); it is simply obtained from the earlier configuration by changing the ordering of the $L_i$. In this new model, the maps $\mu^k$ from (\ref{eq:mu_k_order}) are now:
$$
\mu^3(u,u,u)=1,\ \mu^3(u,u,1)=0,\ \mu^3(u,1,u)=u, \ \mu^3(1,u,u)=u.
$$
The unique disk contributing to $\mu^3(u,u,u)$ is shown in Figure~\ref{fig:Circle_A_infty}(b) by gray shade. It is an immersed disk, and the domain over which it self-overlaps has darker shade. Note that in this model, the degree-one generators $u\in CF^1(L_1,L_2)$, $CF^1(L_3,L_4)$, $CF^1(L_1,L_4)$  correspond to the intersection points on the upper patch, and the degree-one generator $u\in CF^1(L_2,L_3)$ corresponds to the intersection point on the lower patch. We see that we again get $\mu^3(1+u,1+u,1+u)=u$.

The existence of the Massey product above crucially required $\chr\k=2$, because otherwise we would not get $\mu^2(1+u,1+u)=0$, which is necessary to speak of the triple Massey product of $1+u$ with itself. If $\chr\k\neq 2$, then 
$$HF^*(L_1,L_1)\cong \k[u]/(u^2-1)\cong \k[u]/(u-1)\oplus \k[u]/(u+1)$$ is a direct sum of fields, whose Hochschild cohomology as an ordinary algebra vanishes \cite{Ka86} in degree 2, in contrast to the case $\chr \k=2$. So any \ai algebra over $\k[u]/(u-1)\oplus \k[u]/(u+1)$ is formal by \cite{Kade88}  or \cite[Section 3]{Sei15}, in particular the \ai algebra of the equator on $S^2$ is formal. For example, the product $\mu^3(1+u,1+u,1+u)$ can be made to vanish after a formal diffeomorphism. Because of the non-trivial Massey product in characteristic 2, such a formal diffeomorphism, say over $\mathbb{Q}$, will necessarily involve division by 2, and cannot be realised by any geometric choice of the push-offs $L_i$.

In comparison, the topological \ai algebra of the circle  is formal over a field of any characteristic. Indeed, the topological \ai algebra is $\Z$-graded, so if we make this algebra to be based on the cohomology ring $H^*(S^1)\cong\k[x]/x^2$ where $|x|=1$, the only possibly non-trivial products will be $\mu^k(x,\ldots,x,1,x,\ldots,x)$ for grading reasons. On the other hand, every \ai algebra is quasi-isomorphic to a minimal, strictly unital one over a field of any characteristic \cite[Lemma~2.1]{SeiBook08}, \cite[Theorem~3.1.1]{Le02}. In a minimal strictly unital model,  those products  vanish by definition when $k\ge 3$.


\section{The closed-open map for monotone toric fibres}
\label{sec:toric_fibre}

\subsection{The mechanism of Theorem~\ref{th:CO_invt_lag} for toric fibres}
Let $X$ be an $n$-dimen\-sional compact toric Fano variety, and $T\subset X$ the unique monotone toric fibre.
Evans and Lekili\ \cite{EL15} proved (after this paper had appeared as preprint)
that, if $\chr\k= 0$,
the Fukaya category $\Fuk(X)_w$ is split-generated
by several copies of $T$, equipped with the local systems corresponding to the critical points of the Landau-Ginzburg superpotential with critical value $w\in\k$. 
We shall recall the formula for the superpotential of a toric manifold in the next subsection; the common general references are
\cite{CO06,Au07,FO310b}.

Prior to \cite{EL15},  the split-generation by toric fibres had been proved only in the case when the superpotential is Morse, see~Ritter \cite{Ri16}. (For Ritter, proving split-generation requires considerable effort even in the Morse case, if $W$ has several critical points with the same critical value. However, the difficulty is mainly related to the fact that he allows some non-compact toric varieties, where the injectivity of $\CO^*$ is no longer a criterion for split-generation and one must look at $\mathcal{OC}^*$ instead. If we work with compact manifolds, checking that $\CO^*$ is injective for an arbitrary Morse potential is easy: see Corollary~\ref{cor:W_Morse_split_gen}).
An example of a toric Fano variety with non-Morse superpotential over $\C$ has been obtained by Ostrover and Tyomkin \cite{OT09}, and one can check that the superpotential in their case has an $A_3$ singularity.

To complete the literature overview, we should mention the work in progress by Abouzaid, Fukaya, Oh, Ohta and Ono \cite{AFO3} that will prove the split-generation result for toric manifolds that are not necessarily Fano.

Because the toric fibre $T$ is invariant under all the Hamiltonian loops coming from the torus action, it is an obvious example where Theorem~\ref{th:CO_invt_lag} can be put to the test. It turns out that it does allow to prove split-generation away from the Morse case, though not too far from it: the superpotential is required to have at worst $A_2$ singularities, and an extra condition $\chr \k\neq 2,3$ is required, see Corollary~\ref{cor:W_A_2_split_gen}. 

Although  our result is much weaker than the general one from \cite{EL15}, we find it interesting in our approach that the ability to solve  equation $(*)$ from Theorem~\ref{th:CO_invt_lag} depends on whether $W$ is Morse or not. Equip $T$ with a local system $\rho$ which corresponds to a critical point of $W$, then $(T,\rho)$ is wide and we can identify the vector spaces $HF^*(T,\rho)\cong H^*(T)$ via the PSS map $\Phi$. For convenience, let us rewrite  equation $(*)$ from Theorem~\ref{th:CO_invt_lag}:
\begin{equation*}
\tag{$*$}
\mu^2(a,y)+\mu^2(y,a)=\rho(l)\cdot\langle y,l\rangle\cdot 1_T \quad\text{for\ each}\quad y\in H^1(T).
\end{equation*}
We are using the standard spin structure on $T$ hence the sign
 $(-1)^{\epsilon(l)}$ is positive, see Remark~\ref{rem:toric_positive}.
Recall that Theorem~\ref{th:CO_invt_lag}(b) can be applied if there exists no $a\in HF^*(T,\rho)$ making $(*)$ hold.
The Floer cohomology algebra of $(T,\rho)$ is a Clifford algebra determined by the Hessian of $W$ at the point $\rho$, so  the left-hand side of $(*)$ is equal to $\Hess_\rho W(a,y)\cdot 1_T$, at least when $a\in H^1(T)$; we are using  informal notation for the moment. Therefore, finding an element $a$ solving $(*)$ reduces to finding an $a$ such that 
\begin{equation}
\label{eq:solving_eqn_star_hess}
\Hess_\rho W(a,-)=\mathrm{const} \cdot \langle -,l\rangle.
\end{equation}
The ability to find such an $a$ depends on how degenerate $\Hess_\rho W$ is. If $\rho$ is a Morse point of $W$, such an $a$ can always be found, so Theorem~\ref{th:CO_invt_lag}(b) does not apply. However, the Morse case can actually be covered by Theorem~\ref{th:CO_invt_lag}(a), as we explain below. On the other hand, when $\Hess_\rho W$ has kernel, we will have some elements $l\in H_1(T)$ for which equation (\ref{eq:solving_eqn_star_hess}) has no solution $a$. If we consider the $S^1$-action whose orbit is such an element $l$, Theorem~\ref{th:CO_invt_lag}(b) can be applied to the Seidel element of this $S^1$-action to get some new information on $\CO^*$ which is not seen by $\CO^0$.
This information turns out to be sufficient only when the superpotential has $A_2$ singularities, however, there is a possible way of improvement which we speculate upon in the end of this section.

\subsection{The results}
Recall \cite{Cho05,CO06,FO310b,FO3Book}
that the Landau-Ginzburg superpotential of $X$ is a Laurent polynomial $W\co (\k^\x)^n\to \k$ is given by
$$
W(x_1,\ldots,x_n)=\sum_{e}\sum_{j=1}^n x_j^{e^j}
$$
where the first sum is over the outer normals $e\in\Z^n$ to the facets of the polyhedron defining $X$, and $e^j\in \Z$ are their co-ordinates. 
(Sometimes, the superpotential is written down with a Novikov parameter, but we can ignore it because we will only be working with the monotone torus $T$.)
We identify $(\k^\x)^n$ with the space
of all local systems $H_1(T;\Z)\to \k^\x$. For $\rho\in(\k^\x)^n$, we write
$(T,\rho)$ for the torus equipped with this local system.
Also, we will abbreviate $$HF^*(T,\rho)\coloneqq HF^*((T,\rho),(T,\rho)),$$ and the same for Hochschild cohomology.
It is known, see for example \cite[Proposition 3.3]{OT09}, that
\begin{equation}
\label{eq:QH_toric}
QH^*(X)\cong \k[x_1^{\pm 1},\ldots,x_n^{\pm 1}]/Jac(W)=\O(Z),
\end{equation}
where the Jacobian ideal $Jac(W)$ is generated by $(\bd W/\bd x_1,\ldots,\bd W/\bd x_n)$, and
$Z$ is the subscheme of $\Spec \k[x_1^{\pm 1},\ldots,x_n^{\pm 1}]$ defined by the ideal sheaf $Jac(W)$. Then $Z$ is a 0-dimensional scheme supported at the critical points of $W$,
$$\{\rho_1,\ldots\rho_q\}=\Crit W,\quad \rho_i\in (\k^\x)^n.$$
The obstruction number of the torus is given by
$$
w(T,\rho)=W(\rho)\in \k.
$$
Under the isomorphism (\ref{eq:QH_toric}), the quantum product is the usual product on $\O(Z)$, and the first Chern class of $X$ is given by the function $W$ itself. The generalised eigenspace decomposition with respect to $-*c_1(X)$ is simply the decomposition into the local rings at the points  $\rho_i$:
$$
\k[x_1^{\pm 1},\ldots,x_n^{\pm 1}]/Jac(W)\cong \bigoplus_{\rho_i\in \Crit W}\O_{\rho_i}(Z),
$$
the eigenvalue of the $\rho_i$-summand being the critical value $W(\rho_i)$. From Lemma~\ref{lem:eigenvalue_split}, we see that $HF^*(T,\rho)=0$ if $\rho\notin \Crit W$. On the other hand, it is known that $(T,\rho_i)$ is wide for $\rho_i\in \Crit W$, i.e. $HF^*(T,\rho_i)$ is isomorphic as a vector space to $H^*(T)$.

\begin{lemma}
\label{lem:CO_0_torus_fibre}
 Under the isomorphism (\ref{eq:QH_toric}), the map $\CO^0\co QH^*(X)\to HF^*(T,\rho_i)$ is given by
$$\CO^0(f)=f(\rho_i)\cdot 1_T.$$
Here $f(x_1,\ldots x_n)\in QH^*(X)$, $f(\rho_i)\in \k$ is the value of the function at $\rho_i\in \Crit W$, and $1_T\in HF^*(T,\rho_i)$ is the unit. 
\end{lemma}

\begin{proof}
Because $\CO^0$ is a map of algebras, it suffices to prove the lemma when $f=x_k$ is a linear function, $1\le k\le n$. By \cite{MDT06}, $f=\S(\gamma)$ for a Hamiltonian loop $\gamma$ coming from the Hamiltonian torus action, such that the value of the local system $\rho_i$ on an orbit of $\gamma$ equals  the $k$th co-ordinate $\rho_i^k$, which is the same as the value $f(\rho_i)$. So $\CO^0(f)=f(\rho_i)\cdot 1_T$ by Theorem~\ref{th:CO_invt_lag}(a).
\end{proof}

\begin{corollary}
\label{cor:CO_split_torus_fibre}
 For $\rho_i\neq \rho_j\in \Crit W$, the map $\CO^*|_{\O_{\rho_i}(Z)}\to HH^*(T,\rho_j)$ vanishes.
\end{corollary}

\begin{remark}
If $W(\rho_i)\neq W(\rho_j)$, Corollary~\ref{cor:CO_split_torus_fibre} follows from Lemma~\ref{lem:eigenvalue_split}. When  $W(\rho_i)=W(\rho_j)$, the statement is implicit in \cite[Proof of Theorem~6.17]{Ri16} where it is shown that, dually, 
$$\OC\co HH_*(T,\rho_j)\to QH^*(X)$$ hits at most one summand of the form $\O_{\rho_i}$, and we know by Lemma~\ref{lem:CO_0_torus_fibre} that this summand must actually be $\O_{\rho_j}$. The proof in \cite{Ri16} is very different and relies on the variation of the symplectic form.
\end{remark}

\begin{proof}
Let $f\in \k[x_1^{\pm 1},\ldots,x_n^{\pm 1}]$ be such that $f(\rho_i)\neq 0$ and $f(\rho_j)=0$. Then, as an element of $\O_{\rho_i}(Z)$, $f$ is invertible. If the corollary does not hold, $\CO^*(f)$ is also invertible. On the other hand, $\CO^0(f)=0\in HF^*(T,\rho_j)$ by Lemma~\ref{lem:CO_0_torus_fibre}. 
The map $HH^*(T,\rho_j)\to HF^*(T,\rho_j)$, which takes a Hochschild cochain to its  zeroth-order term, is a map of unital algebras, by the formula for the Yoneda product and because the Hochschild cohomology unit is represented by a cochain whose zeroth-order term is the Floer cohomology unit (this follows, for example, from the unitality of $\CO^*$). We have determined that $f$ lies in the kernel of $HH^*(T,\rho_j)\to HF^*(T,\rho_j)$, but that contradicts the fact that $f$ is invertible.
This implies the corollary.
\end{proof}

For $w\in \k$, denote
$$
\Crit_w W=\{\rho\in \Crit W\co  W(\rho)=w\}
$$
the set of all critical points of $W$ with the same critical value $w$.
We will sometimes denote the restrictions of $\CO^0$ and $\CO^*$ to subalgebras of $QH^*(X)$ by the same symbol, when it is otherwise clear that we are considering a restriction.

\begin{corollary}
\label{cor:W_Morse_split_gen}
If $\chr\k\neq 2$, the map 
$$\CO^0\co QH^*(X)_w\longrightarrow \bigoplus_{\rho_i\in\Crit_w W} HF^*(T,\rho_i)$$
is injective if and only if all points of $\Crit_w W$ are Morse.
\end{corollary}

\begin{proof}
By Corollary~\ref{cor:CO_split_torus_fibre}, $\CO^0$ is injective if and only if  its restrictions 
 $\CO^0\co \O_{\rho_i}(Z)\to HF^*(T,\rho_i)$ are injective for each $\rho_i$.
The map $\O_{\rho_i}(Z)\to\k$ which takes $f\in \O_{\rho_i}(Z)$ to its value $f(\rho_i)$ is injective if and only if $\O_{\rho_i}(Z)$ is a field, which is equivalent to the fact that $\rho_i$ is a Morse point of $W$ when $\chr\k\neq 2$. Now apply Lemma~\ref{lem:CO_0_torus_fibre}.
\end{proof}

\begin{proposition}
\label{prop:CO_inject_W_A2}
 Suppose $\chr\k\neq 2,3$ and $W$ has an $A_2$ singularity at a point $\rho$, then $\CO^*\co \O_{\rho}(Z)\to HH^*(T,\rho)$ is injective.
\end{proposition}

\begin{proof}
After an integral linear change of co-ordinates, we may assume that the Hessian of $W$ at $\rho$ is the diagonal matrix:
$\Hess_\rho W=\mathrm{diag}(1,\ldots, 1,0)$. 
We claim that $\O_{\rho}(Z)$ is generated, as a vector space, by the two elements $1$ and $x_n$, where the linear function $x_n$ corresponds to the kernel of $\Hess_\rho W$. Indeed, after a further non-linear change of coordinates with the identity linear part, we can bring $W$ to the canonical form
$$W(\tilde x_1,\ldots, \tilde x_n)=W(\rho)+\sum_{i=1}^{n-1}(\tilde x_s-\rho^i)^2\ + (\tilde x_n-\rho^n)^3.
$$
Here $\rho^i\in\k$ are the co-ordinates of $\rho$. Then $Jac(W)=((\tilde x_1-\rho^1),\ldots,(\tilde x_{n-1}-\rho^{n-1}),(\tilde x_n-\rho^n)^2)$, so $\O_{\rho}(Z)$ is generated, as a vector space, by $1$ and $\tilde x_n$. Because $x_n$, as a function of $\tilde x_1,\ldots,\tilde x_n$, equals $\tilde x_n$ plus terms of order at least 2, it is easy to see that the elements $1,x_n$ also generate the vector space $\O_{\rho}(Z)$.

Let us identify $HF^*(T,\rho)$ with $H^*(T)$ via the PSS map $\Phi$.
Recall that, in general, $HF^*(T,\rho)$ is the algebra generated by $y_1,\ldots,y_n\in H^1(T)$ with relations
$$
y_py_q+y_qy_p=\bd^2_{x_px_q}W(\rho).
$$
In particular, in our case we get $y_py_n+y_ny_p=0$ for any $1\le p\le n$, so $y_n\in HF^1(T,\rho)$ anti-commutes with any element of $HF^*(T,\rho)$ of odd degree.
Consequently, the left hand side of equation $(*)$ from Theorem~\ref{th:CO_invt_lag} vanishes if we put $y=y_n$, and allow $a$ to be of arbitrary odd degree.

Returning to our generator $x_n\in \O_{\rho}(Z)$, we have $x_n=\S(\gamma)$ for a Hamiltonian $S^1$-action (coming from the toric action) such that the element $y_n\in HF^1(T,\rho)$ is dual to the orbit $l\in H_1(T)$ of $\gamma$, so that 
$\langle y_n,l\rangle=1$. So if we put $y=y_n$, the right hand side  of equation $(*)$ from Theorem~\ref{th:CO_invt_lag} becomes
$\rho^n\cdot 1_T\neq 0$. Hence $(*)$ has no solution, and Theorem~\ref{th:CO_invt_lag}(b) says that $\CO^*(x_n)$ and $1_{HH}=\CO^*(1)$ are linearly independent.
\end{proof}

Combining the above discussion with the split-generation criterion, we get the following corollary.

\begin{corollary}
\label{cor:W_A_2_split_gen}
 Suppose $\chr\k\neq 2,3$ and each critical point $\rho_i\in \Crit_w W$ is either Morse or an $A_2$ singularity. Then the copies of the monotone toric fibre with local systems $\{(T,\rho_i)\}_{\rho_i\in\Crit_wW}$ split-generate $\Fuk(X)_w$.\qed
\end{corollary}
 
\subsection{A  way of extending Theorem~\ref{th:CO_invt_lag}}
It is in fact not surprising that Theorem~\ref{th:CO_invt_lag} turned out to be efficient only for $A_2$ singularities. The main result on which Theorem~\ref{th:CO_invt_lag} is based upon is Proposition~\ref{prop:compute_CO_1}, which computes the {\it linear} part $\CO^1$ of the closed-open map, while
the only non-Morse singularity whose local Jacobian is generated as a vector space by constant and {\it linear} functions is the $A_2$ singularity (for which the Jacobian is generated by $1$ and $x_n$ as above). One could extend the computation in Proposition~\ref{prop:compute_CO_1} to all orders of $\CO^*$ when applied to products of 1-cochains on $L$; we conjecture that the following holds.

\begin{conjecture}
The restriction 
$$\CO^k(\S(\gamma))|_{CF^1(L,L)^{\otimes k}}\co CF^1(L,L)^{\otimes k}\to CF^0(L,L)$$
equals
\begin{equation}
\label{eq:compute_CO_k}
(-1)^{\epsilon(l)}\rho(l)\cdot (l^*)^{\otimes k}\cdot 1_L
\end{equation}
on symmetrised tensor products of Floer 1-cocycles.
Here $l^*\co CF^1(L,L)\to \k$ is given by $l^*(x)=\langle \Psi(x),l\rangle$, and $l\in H_1(L)$ is the orbit of $\gamma$. 
\end{conjecture}

\begin{remark}
As in Proposition~\ref{prop:compute_CO_1}, part of the statement is that the image of this restriction necessarily lands in $CF^0(L,L)$: this follows for degree reasons.  Although the proof of the above formula should be analogous to Proposition~\ref{prop:compute_CO_1}, one new issue arises which we have not checked in detail. Consider the moduli spaces $\M^\gamma(x_1,\ldots,x_k;x_0)$ from Section~\ref{sec:main_proof} and the pearly moduli spaces analogous to Figure~\ref{fig:loop_count} but with more inputs. The new issue is a different type of domain degenerations coming from the collision of input points: e.g.~several punctured inputs for a curve in $\M^\gamma(x_1,\ldots,x_k;x_0)$ may collide and create a bubble. To prove (\ref{eq:compute_CO_k}), one would need to argue that these collisions cancel out when the input string is symmetrised.
\end{remark}

Formula~(\ref{eq:compute_CO_k}) is a chain level computation, and whether it survives to something non-trivial in Hochschild cohomology will be governed by equations generalising equation $(*)$ from Theorem~\ref{th:CO_invt_lag}; those equations will be determined by the \ai structure maps on $L$ up to order $k+1$. 
When $L$ is the monotone toric fibre,
the \ai structure maps have been related to higher-order partial derivatives of $W$  by Cho \cite{Cho05}, and intuitively, the more degenerate the superpotential is, the more non-trivial information from (\ref{eq:compute_CO_k}) survives to Hochschild cohomology. Consequently, these observations are a possible starting point for proving   split-generation results for toric Fano varieties with other degenerate superpotentials. However, further development of this discussion seems both complicated and not particularly demanded, given the general results of \cite{AFO3,EL15}.

\bibliography{Symp_bib}{}
\bibliographystyle{plain}

\end{document}